\documentclass[12pt,a4paper]{amsart}

\usepackage{hyperref}
\usepackage{cite}

\usepackage{graphicx}
\usepackage{xcolor}

\usepackage{amsthm} 
\usepackage[leqno]{amsmath}
\usepackage{tikz}
\makeatletter
\newcommand{\leqnomode}{\tagsleft@true}
\newcommand{\reqnomode}{\tagsleft@false}
\makeatother

\newtheorem{theorem}{Theorem}
\newtheorem{lemma}{Lemma}
\newtheorem{prop}{Proposition}
\newtheorem{claim}{Claim}

\begin{document}

\title{Search for invariant sets of the generalized tent map}
\author[Ayers]{Kimberly Ayers}
\address{Cal State San Marcos, San Marcos, CA}
\author[Dmitrishin]{Dmitriy Dmitrishin}
\address{Odessa National Polytechnic University}
\author[Radunskaya]{Ami Radunskaya}
\address{Pomona College, Claremont, CA}
\author[A.Stokolos]{Alexander Stokolos}
\address{Georgia Southern University}
\author[C.Stokolos]{Constantine Stokolos}
\address{Odessa National Polytechnic University}

\begin{abstract}
This paper describes a predictive control method to search for unstable periodic orbits of the generalized tent map. The invariant set containing periodic orbits  is a repelling set with a complicated Cantor-like structure.  Therefore,
a simple local stabilization of the orbit may not be enough to find a periodic orbit,  due to the small measure of the basin of attraction. It is shown that for 
certain values of the control parameter, both the local behavior and the global behavior of solutions change in the controlled system; in particular, the invariant set enlarges to become an interval or the entire real axis.  The computational particularities of using the control system 
are considered, and necessary conditions for the orbit to be periodic are given. The question of local asymptotic stability of 
subcycles of the controlled system's stable cycles is fully investigated, and some statistical properties of the subset of the classical 
Cantor middle thirds set that  is determined by the periodic points of the generalized tent map are described.

\medskip
\noindent \textit{Keywords.} Generalized tent map, predictive control, periodic orbits, local stabilization, invariant sets.  
\end{abstract}

\maketitle

\section*{Introduction}

In his seminal paper \textit{R. Lozi} noted that numerical computations using computers play a central role in analyzing 
solutions of nonlinear dynamical systems, that computer-aided proofs are complex and necessarily require additional special 
validation of the results \cite{Lozi}. Nevertheless, numerous studies in fields related to chaotic dynamical systems are confident 
in the numerical solutions that they found using popular software, sometimes without carefully checking the reliability of these results.
Computationally, computers store numbers in registers and memory cells with a limited number of digits.  Thus, the set of real numbers represented in the machine is discrete and finite - irrational numbers, and rational numbers with infinite decimal expansions are rounded to decimal expansions that terminate.  This can lead to problems when attempting to numerically find unstable periodic orbits of discrete chaotic dynamical systems.  

Consider the family of chaotic dynamical systems given by the \emph{generalized tent map}:
\begin{equation}\label{eq}
x_{n+1}=f(x_n),\;n=1,\,2,\ldots,
\end{equation}
where 
\begin{equation}\label{tent}
f(x)=H\left(\frac12-\left|x-\frac12\right|\right)=
\left\{\begin{array}{ll}
Hx,\,x\leq\frac12, \\
H(1-x),\,x>\frac12,
\end{array}\right.
\end{equation}
$x\in(-\infty, \; +\infty),$ $H\geq2.$ Note that, despite the relative simplicity of function~(\ref{tent}), 
Equation~(\ref{eq}) is of great theoretical importance and has appeared in several applications \cite{Derr, Goh}.  Consider the classical tent map, given by $H=2$. Note that $x_0=\frac23$ is a fixed point: $f(\frac23)=\frac23$.  Since $|f'(\frac23)|=2>1$, $\frac23$ is an unstable fixed point.  Thus, any amount of rounding error will cause an orbit to eventually diverge from this fixed point, and we can see this computationally. 

In this article, we show how we can correct the computational procedure in the problem of finding unstable periodic orbits 
of a nonlinear discrete system using the tent map as an illustrative example.

The dynamics of even the simplest nonlinear discrete systems can be quite complex \cite{Devaney, Chen, Andr}. Such systems are often characterized 
by extremely unstable motions in phase space, which are defined as chaotic \cite{Devaney}. In dissipative systems, these motions define invariant sets, 
which can be strange attractors or repellers. Trajectories on such invariant sets have positive Lyapunov exponents; therefore, these 
trajectories are exponentially sensitive to the initial conditions. 
Unstable periodic orbits are canonical examples of repelling invariant sets, and we can get insight into the general study of repellers by considering 
 unstable periodic orbits. 
 Periodic orbits have a hierarchical structure determined by their length, which 
makes it possible to calculate various characteristics of invariant sets and their subsets, for example, topological dimension and entropy \cite{Biham}. 
However, when applying numerical methods to search for points along a periodic orbits  we encounter a number 
of fundamental problems. Due to sensitivity to initial conditions and rounding errors, after several calculation steps the results can 
vary greatly depending on the chosen calculation accuracy: the so-called ``butterfly effect'' occurs. Even with a real possibility to choose 
a very high accuracy of calculations, we will never be able to say with certainty what we actually found: a long cycle, a pseudo-cycle or 
a strange attractor \cite{Lozi}.

There are several methods of searching for periodic orbits of a nonlinear discrete system, which can be divided into two groups: methods 
that do not use the correction of the original discrete dynamical system, for example, the method of interval arithmetic analysis \cite{Zyg, Gal2002, Gal2001}, 
or the method connected with the construction of special Hamiltonian systems \cite{Biham}; and methods based on local stabilization of 
an unknown periodic orbit of a given length \cite{Ott, Qian, Aleks, Yang, Miller, DSLS}. The second group of methods is more preferable in the sense that their accuracy increases with the number of iterations, due to the correction of the original dynamical system. If  we 
can locally stabilize the orbit with the help of the control action, then the trajectories of the system will remain in the neighborhood of the orbit and will be attracted to it, 
i.e. the periodic orbit will be ``found". By choosing different initial points, different periodic orbits can be found. To solve the problems of 
stabilization in the search of periodic orbits,  various control schemes were proposed that use information about the states of the controlled 
system at previous points in time \cite{Pyr, Vie, DK, Mor} (delayed control) or about the states of the initial system at future points in time \cite{Polyak, Ushio, Shal, DSI} 
(predictive control). 

The purpose of this paper is to illustrate the effectiveness of the predictive control method using the example of the generalized tent map. 
The invariant sets of the generalized tent map are repellers that have the structure of a Cantor set. We show that in a controlled system, 
along with a change in the local behavior of solutions, the global behavior also changes. We describe the invariant sets of the controlled system, 
which are either a single interval or a union of intervals. Locally asymptotically stable periodic orbits are subsets of the invariant set. The basins of 
attraction of these orbits are also discussed. 

The paper is organized as follows. In Section \ref{sec:background} we present the main result from previous work \cite{DSI}, which substantiates the generalized 
predictive control scheme. Section \ref{sec:ControlSystem} poses the problem of finding a cycle of length $T$, and provides necessary and sufficient conditions
for the local asymptotic stability of this cycle. 

Section \ref{sec:ControlSystem} presents the main theoretical results in the paper:  it provides conditions under which all solutions 
of the controlled system are bounded (Theorem \ref{t2}), or solutions with initial values from a given interval are bounded (Theorem \ref{t3}). We also show that the cycles do not lose the property of local asymptotic stability. In Section \ref{sec:ComputationalParticularities}, some computational particularities of using 
the control system for finding cycles are considered, and  necessary conditions for the found sequence to be a cycle are given. 
In Section \ref{sec:GraphicalStudy}, using the example of stabilization of cycles of small lengths, we consider the relationship between the graphical properties of the original 
map and the map determined by the controlled system. In Section \ref{sec:Subcycles}, the question about local asymptotic stability of subcycles of the controlled 
system's stable cycles is fully investigated (Theorem \ref{t4}). Section \ref{sec:DistributionCyclicPoints} studies subsets of the classical Cantor middle thirds set. Since any point in the Cantor set is a limit of periodic points, we can use the controlled system to visualize the Cantor set itself.
We compare the distribution of periodic points of very long cycles to the distribution of points of the first type, i.e. endpoints of the intervals that make up the complement of the Cantor set.  The images in this paper were created using Maple, and the files that generated them can be found at \cite{code}. 

\section{Background} \label{sec:background}

We consider the discrete system 
\begin{equation}\label{sys}
x_{n+1}=f(x_n),\;x_n\in {\mathbb R}^m,\;n=1,\,2,\ldots,
\end{equation}
where $f(x)$ is, in general, a nonlinear differentiable function from ${\mathbb R}^m$ to $\mathbb{R}^m$.  It is assumed that this system has one or more unstable 
$T$-cycles $\left(\eta_1,\ldots,\,\eta_T\right),$ where all the vectors $\eta_1,\ldots,\,\eta_T$ are distinct, i.e. $\eta_{j+1}=f(\eta_j),$ 
$j=1,\ldots,\,T-1,$ $\eta_1=f(\eta_T).$ Vectors in the cycles are called {\it cyclic points}, and each cycle of length $T$ constitutes a {\it $T$-periodic orbit}. The {\it multipliers} of the 
unstable cycles are defined as the eigenvalues of the products of Jacobi matrices $\prod\limits_{j=1}^T{f'(\eta_{T-j+1})}$ of dimension 
$m\times m$ at the points of the cycle. The matrix $\prod\limits_{j=1}^T{f'(\eta_{T-j+1})}$ is called the Jacobi matrix of the cycle 
$\left(\eta_1,\ldots,\,\eta_T\right).$ The collection of all multipliers $\left\{\mu_1,\ldots,\,\mu_m\right\}$ is called the spectrum of 
the Jacobi matrix. As a rule, the cycles $\left(\eta_1,\ldots,\,\eta_T\right)$ of system~(\ref{sys}) are not known a priori. Consequently, 
the spectrum is not known either.

Consider the control system	
\begin{equation}\label{csys}
x_{n+1}=f\left(\vartheta_1 x_n + \sum\limits_{j=2}^N{\vartheta_j f^{\left((j-1)T\right)}}(x_n)\right).
\end{equation}
The numbers $\vartheta_1,\ldots,\,\vartheta_N$ are real. We can verify that when $\sum\limits_{j=1}^N{\vartheta_j}=1$, 
system~(\ref{csys}) also has a cycle $\left\{\eta_1,\ldots,\,\eta_T\right\}.$ The problem is to choose a parameter $N$ 
and coefficients $\vartheta_1,\ldots,\,\vartheta_N$ so that the cycle $\left\{\eta_1,\ldots,\,\eta_T\right\}$ of system~(\ref{csys}) 
is locally asymptotically stable.

The following result gives a criterion for stability of the cycles in terms of the multipliers. We will see how to use this result to determine a suitable range for the control parameter, $\vartheta$.
\begin{prop}[\cite{DSI} ]\label{Prop1}
Suppose $f\in C^1$ and that system~(\ref{sys}) has an unstable $T$-cycle with multipliers $\left\{\mu_1,\ldots,\,\mu_m\right\}.$
Then this cycle will be a locally asymptotically stable cycle of system~(\ref{csys}) if 
$$
\mu_j \left[r(\mu_j)\right]^T \in D,\;j=1,\ldots,\,m,
$$
where $D=\left\{z\in C:\;|z|<1\right\}$ is an open central unit disc on the complex plane, $r(\mu)=\sum\limits_{j=1}^N{\vartheta_j \mu^{j-1}}.$
\end{prop}

A set $U$ is called {\it invariant} for equation~(\ref{sys}) if for any $x_0\in U$ it follows that $f^{(k)}(x_0)\in U,$ $k=1,\,2,\ldots\,.$ It is shown in \cite{Leonov}
that for $H=3$ the invariant set of equation~(\ref{eq}) is the classical Cantor set. Analogously, it can be shown that when $H>2,$ the invariant set for Equation~(\ref{eq}) 
is a set of the Cantor type, that is an uncountable, closed set with zero Lebesgue measure. Note that each point of the invariant 
set can be represented in the form $\sum\limits_{j=1}^{\infty}{\frac{\alpha_j}{H^j}},$ where $\alpha_j\in\{0,\,H-1\}.$ This set includes a countable subset 
of all periodic points of map~(\ref{tent}). If $x_0$ does not belong to the invariant set, then the corresponding sequence 
$\left\{f^{(k)}(x_0)\right\}_{k=1}^{\infty}$ tends to $-\infty.$ Such invariant sets are called repellers of map~(\ref{tent}). 
The problem we address in this paper is the following: 
for given $T$, can we numerically find $T$-periodic points of the tent map~(\ref{tent}) as limit points of some iterative scheme.

\section{Main results: the control system and global behavior of the control system trajectories} \label{sec:ControlSystem}

To solve the problem of finding unstable periodic orbits of the tent map, we will use the predictive control method, which we now describe explicitly.  Along with equation~(\ref{eq}), consider the equation
\begin{equation}\label{ceq}
x_{n+1}=F(x_n),\;n=1,\,2,\ldots,
\end{equation}
where $F(x)=f\left(\vartheta x + (1-\vartheta)f^{(T)}(x)\right),$ $\vartheta$ is some real number, called the {\it control parameter}, whose value will be determined later.
We will call Equation~(\ref{ceq})  the {\it control system} for equation~(\ref{eq}). The idea is to find the values of $\vartheta$ that stabilize our target periodic cycles.

Let $\left\{\eta_1,\ldots,\,\eta_T\right\}$ be a cycle of  the equation~(\ref{eq}) of length $T$. Since $\vartheta \eta_k + (1-\vartheta)f^{(T)}(\eta_k)=\eta_k,$ then 
$F(\eta_k)=f(\eta_k),$ which means that the cycle of the equation~(\ref{eq}) will also be a cycle of equation~(\ref{ceq}). Note that the converse statement 
is generally not true. 

The multiplier of the equation~(\ref{eq}) cycle is defined by the formula
$$
\mu = f'(\eta_T)\cdot\ldots\cdot f'(\eta_1). 
$$
Since $\left|f'(\eta_j)\right|=H,$ then $|\mu|=H^T>1,$ that is, any cycle of equation~(\ref{eq}) is unstable. Let us find the value of the multiplier 
$\lambda$ of the same cycle $\left\{\eta_1,\ldots,\,\eta_T\right\},$ but for equation~(\ref{ceq}). From Proposition \ref{Prop1} we get that 
$$
\lambda=\mu \left(\vartheta+(1-\vartheta)\mu\right)^T. 
$$

In what follows, we will consider two cases separately: $\mu>0$ and $\mu<0.$ Let $\mu=H^T.$ Then the condition for local asymptotic stability 
of the $T$-cycle of equation~(\ref{ceq}) is: $|\lambda|=\left|H^T \left(\vartheta + (1-\vartheta)H^T\right)^T\right|<1,$ 
from which it follows that 
\begin{equation}\label{inp}
1<
\frac{H^T-\frac{1}{H}}{H^T-1}<
\vartheta<
\frac{H^T+\frac{1}{H}}{H^T-1}<
\frac52.
\end{equation}

If $\mu=-H^T,$ then the condition for local asymptotic stability of the equation~(\ref{ceq}) cycle is: 
$|\lambda|=\left|H^T \left(\vartheta - (1-\vartheta)H^T\right)^T\right|<1,$
from which 
\begin{equation}\label{inn}
\frac12<
\frac{H^T-\frac{1}{H}}{H^T+1}<
\vartheta<
\frac{H^T+\frac{1}{H}}{H^T+1}<
1.
\end{equation}

Thus, Proposition \ref{Prop1} yields the following conditions for asymptotic stability of cycles:

\begin{claim}\label{l1}
Given a $T$-cycle of equation~(\ref{eq}) with the multiplier $\mu.$ This cycle will be a locally asymptotically stable cycle of 
equation~(\ref{ceq}) if inequalities~(\ref{inp}), in the case $\mu>0,$ or inequalities~(\ref{inn}), when $\mu<0,$ are satisfied.
\end{claim}


If there are locally asymptotically stable cycles in equation~(\ref{ceq}), then the invariant set of this equation also includes the basins of 
attraction of these cycles. Since the basin of attraction of a cycle is an open set, its measure is positive. However, there remains the question 
about how large this measure is. This question is important when doing numerical experiments, because we must  judiciously choose the initial point to be in the
 basin of attraction the desired cycle. Thus, as we pass from equation~(\ref{eq}) to equation~(\ref{ceq}), we are ensured that the nature of the periodic orbit (the set of points of the cycle) changes from a repeller into an attractor.  For numericaly tractability, it is also important 
to  know that we have a basin of attraction of sufficiently large measure.

Next, we will establish the properties of invariant sets of equation~(\ref{ceq}) and the global behavior of its solutions. In particular, we will show that under the conditions on the control parameter, $\theta$, given in Claim \ref{l1}, we have globally attracting invariant sets made up of intervals.  

We state these properties as two theorems, one for the case when the multiplier $\mu$ is positive, and one when $\mu$ is negative.  The proofs of the two theorems are broken down into several lemmas, each dealing with a specific range of $\vartheta$ values. 
The two main theorems follow from these lemmas, allowing us to find all cycles of arbitrary 
lengths with any given accuracy. 

\begin{theorem}\label{t2} ({\it Case $\mu > 0$})

If inequalities $\frac{H^T-\frac{1}{H}}{H^T-1}<\vartheta\leq\frac{H^T+\frac1H}{H^T-1}$ are satisfied, then any solution of equation~(\ref{ceq}) is bounded, and its limit set is contained in $\left[0,\,\frac{H}{2}\right]$.  Moreover, any $T$-cycle $\left\{\eta_1,\ldots,\,\eta_T\right\}$ 
of this equation, for which the quantity $\mu = f'(\eta_T)\cdot\ldots\cdot f'(\eta_1)$ is positive, is locally asymptotically stable.
\end{theorem}

\begin{theorem}\label{t3} ({\it Case $\mu < 0$})
If the inequalities $\frac{H^T-\frac{1}{H}}{H^T+1}<\vartheta\leq\frac{H^T}{H^T+1}$ are satisfied, the set $\left[0,\,\frac{H}{2}\right]$ is an invariant set
of equation~(\ref{ceq}). Moreover, any $T$-cycle $\left\{\eta_1,\ldots,\,\eta_T\right\}$ of this equation, for which the quantity 
$\mu = f'(\eta_T)\cdot\ldots\cdot f'(\eta_1)$ is negative, is locally asymptotically stable.
\end{theorem}
Theorems \ref{t2} and \ref{t3} will be proved through a sequence of five lemmas, each one dealing with a different sub-case.  Lemmas \ref{l2}-\ref{l4} divide the positive multiplier case into three sub-cases: one where $\vartheta$ is exactly in the middle of the allowed interval (Lemma \ref{l2}), one where $\vartheta$ is in the lower half of the allowed interval (Lemma \ref{l3}) and one where $\vartheta$ is in the upper half of the allowed interval (Lemma \ref{l4}).  The remaining two lemmas (Lemmas \ref{l5} and \ref{l6}) deal with the case of a negative multiplier, $\mu < 0$.

\subsection{Proof of Theorem \ref{t2}}
Throughout the proof, to break down the calculations, we let the function
$$ \zeta(x) = \vartheta x + (1 - \vartheta)f^{(T)} (x) $$
be the intermediary linear combination that appears in the definition of the control function, $F(x)$ given in Equation \eqref{ceq}. Thus, $F(x) = f(\zeta(x))$, and $\zeta_0 = \zeta(x_0)$.

Suppose $\mu > 0$, and let inequalities~(\ref{inp}) be satisfied. Divide the conditions~(\ref{inp}) into three cases: (i) $\vartheta=\frac{H^T}{H^T-1},(ii) $
$\frac{H^T-\frac{1}{H}}{H^T-1}<\vartheta<\frac{H^T}{H^T-1},$ and (iii) $1<\frac{H^T}{H^T-1}<\vartheta<\frac{H^T+\frac{1}{H}}{H^T-1}.$. Each case is treated in a separate lemma.

\begin{lemma}\label{l2}
Let $\vartheta=\frac{H^T}{H^T-1}.$ Then, if $x_0\leq 0,$ it follows that $F(x_0)=0,$ hence, $F^{(k)}(x_0)=0$ when $k=1,\,2,\ldots;$
if $x_0\geq 1,$ then $F(x_0)<0,$ hence, $F^{(k)}(x_0)=0$ when $k=2,\,3,\ldots;$ if $x_0\in(0,\,1),$ then 
$\left\{F^{(k)}(x_0)\right\}_{k=2}^{\infty}\in[0,\,1].$  
\end{lemma}

\begin{proof}
Let $x_0\leq 0,$ then $f(x_0)=Hx_0\leq 0,$ $f^{(2)}(x_0)=H^2 x_0\leq 0,$ \ldots, $f^{(T)}(x_0)=H^T x_0.$  Find 
$$
\zeta_0 = \zeta(x_0) =\vartheta x_0 + (1-\vartheta)f^{(T)}(x_0)=\frac{1}{H^T-1}\left(H^T x_0 - H^T x_0\right)=0.
$$ 
Hence, $F(x_0)=f(\zeta_0)=0.$

\noindent
 Let $x_0\geq 1,$ then 
$$f(x_0)=H (1-x_0)\leq 0,f^{(2)}(x_0)=H^2 (1-x_0)\leq 0, \dots, f^{(T)}(x_0)=H^T (1-x_0),$$
and
\begin{eqnarray*}
\zeta_0 =\vartheta x_0 + (1-\vartheta)f^{(T)}(x_0)&=&\frac{1}{H^T-1}\left(H^T x_0 - H^T (1-x_0)\right)\\
&=&\frac{H^T}{H^T-1}(2x_0-1)>1,
\end{eqnarray*}
so that
$F(x_0)=f(\zeta_0)=H(1-\zeta_0)<0.$ Hence, $F^{(2)}(x_0)=0.$ 

\noindent
 The last case remains: if $x_0\in(0,\,1)$, then there is a $k_0$ such that  $F^{(k)}(x_0)=0,$ for all $k>k_0$;   or 
$\left\{F^{(k)}(x_0)\right\}_{k=2}^{\infty}\in(0,\,1)$ for all $k$.
\end{proof}

The case $\frac{H^T-\frac{1}{H}}{H^T-1}<\vartheta<\frac{H^T}{H^T-1}$ is considered in a similar way.

\begin{lemma}\label{l3}
Let $\frac{H^T-\frac{1}{H}}{H^T-1}<\vartheta<\frac{H^T}{H^T-1}.$ Then, \\ if $x_0\leq 0$ or $x_0\geq 1,$ 
${F^{(k)}(x_0)\xrightarrow[k\rightarrow\infty]{}0;}$\\ when $x_0\in(0,\,1),$ either $F^{(k)}(x_0)\xrightarrow[k\rightarrow\infty]{}0$
or $\left\{F^{(k)}(x_0)\right\}_{k=1}^{\infty}\in(0,\,1).$
\end{lemma}

\begin{proof}
Let $x_0\leq 0,$ then $f(x_0)=Hx_0\leq 0,$ \ldots, $f^{(T)}(x_0)=H^T x_0.$ Find 
$$
\zeta_0 =\vartheta x_0 + (1-\vartheta)f^{(T)}(x_0)=\left(H^T - \vartheta (H^T - 1)\right)x_0.
$$
It follows from the inequality $\frac{H^T-\frac{1}{H}}{H^T-1}<\vartheta<\frac{H^T}{H^T-1}$ that, first, $\zeta_0\leq 0,$ 
and 
$$F(x_0)=f(\zeta_0)=H\zeta=H\left(H^T - \vartheta (H^T - 1)\right)x_0\leq 0,$$ and, second: 
$$0<\alpha = H\left(H^T - \vartheta (H^T - 1)\right)<1,$$ from which $F(x_0)=-\alpha|x_0|>-|x_0|.$ Then 
$\left|F^{(k)}(x_0)\right|=\alpha^k |x_0|\xrightarrow[k\rightarrow\infty]{}0.$

Let $x_0\geq 1,$ then $f(x_0)=H(1-x_0)\leq 0,$ \ldots, $f^{(T)}(x_0)=H^T (1-x_0),$

\begin{eqnarray*}
\zeta_0&=&\vartheta x_0 + (1-\vartheta)f^{(T)}(x_0)
\\&=&x_0+(\vartheta-1)\left(x_0 + H^T (x_0-1)\right)\\
&>&x_0>1.
\end{eqnarray*}

Then $F(x_0)=f(\zeta_0)=H(1-\zeta_0) < 0.$ Hence, $$F^{(2)}(x_0)=-\alpha\left|F(x_0)\right|>-\left|F(x_0)\right|, \left|F^{(k)}(x_0)\right|=\alpha^{k-1} \left|F(x_0)\right| \xrightarrow[k\rightarrow\infty]{}0.$$

Let $x_0\in(0,\,1).$ If for some $k_0,$ $F^{(k_0)}(x_0)\leq 0$ or $F^{(k_0)}(x_0)\geq 1,$ then 
$F^{(k)}(x_0)\xrightarrow[k\rightarrow\infty]{}0.$ Otherwise, $\left\{F^{(k)}(x_0)\right\}_{k=1}^{\infty}\in(0,\,1).$
The lemma is proved. 
\end{proof}

The case $\frac{H^T}{H^T-1}<\vartheta<\frac{H^T+\frac{1}{H}}{H^T-1}$ remains.

\begin{lemma}\label{l4}
Let $\frac{H^T}{H^T-1}<\vartheta<\frac{H^T+\frac{1}{H}}{H^T-1}.$ In this case, if $-H^2+\frac{H}{2}\leq x_0 \leq \frac{H}{2},$ 
then
\begin{equation}\label{inl4}
-H^2+\frac{H}{2}\leq F^{(k)}(x_0) \leq \frac{H}{2},\;k=1,\,2,\ldots\,.
\end{equation}
If $x_0>\frac{H}{2}$ or $x_0<-H^2+\frac{H}{2},$ there exists a number $k_0\ge 0$ such that inequalities~(\ref{inl4}) are satisfied 
for all $k$ greater than $k_0.$ 
\end{lemma}

\begin{proof}
Let $x_0\leq 0,$ then $f(x_0)=Hx_0\leq 0,$ \ldots, $f^{(T)}(x_0)=H^T x_0,$
$$
\zeta_0=\left(H^T - \vartheta (H^T - 1)\right)x_0.
$$
The inequality $\displaystyle{\frac{H^T}{H^T-1}<\vartheta<\frac{H^T+\frac{1}{H}}{H^T-1}}$ is equivalent to 
$$-\frac{1}{H}<H^T-\vartheta(H^T-1)<0,$$ from which it follows that $0\leq\zeta_0\leq\frac{|x_0|}{H}.$ If $-H\leq x_0\leq 0,$ then 
$0\leq\zeta_0\leq 1$ and $0\leq F(x_0)\leq \frac{H}{2}.$

If $x_0<-H$ and $\zeta_0\geq 1,$ then $0\geq F(x_0)=H(1-\zeta_0)>H-|x_0|.$ Moreover, if $F^{(k)}(x_0)<-H$ and 
$\left(H^T - \vartheta (H^T - 1)\right)F^{(k)}(x_0)\geq 1,$ then $0\geq F^{(k+1)}(x_0)>H-\left|F^{(k)}(x_0)\right|>(k+1)H-|x_0|.$
It means that there exists such $k_0$ that $-H\leq F^{(k_0)}(x_0)\leq 0.$ And hence, $0\leq F^{(k_0+1)}(x_0)\leq\frac{H}{2}.$

Let $0\leq x_0\leq \frac{H}{2},$ then $H^T \left(1-\frac{H}{2}\right)\leq f^{(T)}(x_0)\leq\frac{H}{2}.$ Since $\vartheta>1,$ then
\begin{eqnarray*}
-\frac34 H &<&\vartheta x_0 -(\vartheta-1)\frac{H}{2} \\
 &\leq& \zeta_0 \\
&\leq& \vartheta x_0 + (\vartheta-1)H^T\left(\frac{H}{2}-1\right)  \\
&<& \vartheta\frac{H}{2} + (\vartheta-1)H^T\left(\frac{H}{2}-1\right) \\
&<&\frac{1}{H^T-1}\left( \left( H^T+\frac{1}{H} \right)\frac{H}{2} + \left( 1+\frac{1}{H} \right) H^T \left( \frac{H}{2}-1 \right) \right)  \\
&=&\frac{1}{H^T-1} \left( H^{T+1} - \frac{H^T}{2} - H^{T-1} + \frac12\right).
\end{eqnarray*}
Note that $1<H-1<\frac{1}{H^T-1} \left( H^{T+1} - \frac{H^T}{2} - H^{T-1} + \frac12\right)<H+\frac12.$ Then 
$$
F(x_0) > H\left(1 - \frac{1}{H^T-1} \left( H^{T+1} - \frac{H^T}{2} - H^{T-1} + \frac12\right) \right) > -H\left(H - \frac12\right).
$$
Therefore, if $0\leq x_0\leq \frac{H}{2},$ then $-H^2+\frac{H}{2}\leq F(x_0) \leq \frac{H}{2}.$  Moreover, if $ F(x_0)<0$ then for some $k_0$ the inequality $0 \leq F^{(k_0)}(x_0) \leq \frac{H}{2}$ is satisfied.

Let $x_0 > \displaystyle{\frac{H}{2}}$, then 
$$f^{(T)}(x_0)=H^T (1-x_0), \quad \zeta_0 =\vartheta x_0 + (\vartheta - 1) H^T (x_0 - 1) > 1, \hbox{ and } F(x_0)<0.$$ 
This means that for some $k_0$ the inequality $0 \leq F^{(k_0)}(x_0) \leq \frac{H}{2}$ is satisfied.  

Summing up the four cases above, we get that when $-H^2+\frac{H}{2}\leq x_0 \leq \frac{H}{2},$ inequalities~(\ref{inl4}) hold. 
If $x_0>\frac{H}{2}$ or $x_0<-H^2+\frac{H}{2},$ then inequalities~(\ref{inl4}) will be satisfied starting from some iterate, $x_k$. 
The lemma is proved. 
\end{proof}

The assertion of Theorem~\ref{t2} follows from Claim~\ref{l1}, and Lemmas \ref{l2}, \ref{l3} and \ref{l4}, .

\subsection{Proof of Theorem \ref{t3}} 

Let us now pass to the study of the global behavior of the $T$-cycles $\left\{\eta_1,\ldots,\,\eta_T\right\}$ of equation~(\ref{ceq}) 
for which the quantity $\mu = f'(\eta_T)\cdot\ldots\cdot f'(\eta_1)$ is negative, assuming that conditions~(\ref{inn}) are satisfied.

\begin{lemma}\label{l5}
Let conditions~(\ref{inn}) be satisfied and $x_0<0.$ Then \\ $F^{(k)}(x_0)\xrightarrow[k\rightarrow\infty]{}-\infty.$  
\end{lemma}

\begin{proof}
Since $x_0<0,$ then $f(x_0)=Hx_0\leq 0,$ \ldots, $f^{(T)}(x_0)=H^T x_0,$ and 
$$
\zeta_0=\vartheta x_0 + (1-\vartheta)f^{(T)}(x_0)=\left(\vartheta + (1 - \vartheta) H^T \right)x_0.
$$
Since $\frac12<\vartheta<1,$ then $\zeta_0<x_0,$ $F(x_0)=H\zeta_0<H x_0,$ $F^{(k)}(x_0)<H^k x_0$ when $k=1,\,2,\ldots,$ 
from which the conclusion of the lemma follows.
\end{proof}

\begin{lemma}\label{l6}
Let the inequalities $\frac{H^T-\frac{1}{H}}{H^T+1}<\vartheta \leq \frac{H^T}{H^T+1}$ be satisfied, and $0\leq x_0 \leq \frac{H}{2}.$ 
Then $0\leq F(x_0)\leq \frac{H}{2}.$  
\end{lemma}

\begin{proof}
Let $1\leq x_0 \leq \frac{H}{2},$ then $f(x_0)=H(1-x_0)\leq 0,$ \ldots, $f^{(T)}(x_0)=H^T (1-x_0),$
$$
\zeta_0=\vartheta x_0 + (1-\vartheta)f^{(T)}(x_0)=\frac{1}{H^T + 1} \left ( (H^T+\alpha)x_0 + (1-\alpha)H^T(1-x_0)\right),
$$
where $\vartheta=\frac{H^T+\alpha}{H^T+1},$ $-\frac{1}{H} < \alpha \leq 0.$ Since $0\leq x_0 \leq \frac{H}{2}$,
\begin{eqnarray*}
\zeta_0 &=& \frac{1}{H^T+1} \left( H^T + \alpha(x_0 + x_0 H^T - H^T)\right) \\
&>& \frac{1}{H^T+1} \left( H^T - \frac{1}{H} \left(\frac{H}{2} + \frac{H}{2} H^T - H^T\right)\right) \\
&=& \frac{1}{H^T+1} \left( \frac12 H^T - \frac12 + H^{T-1} \right) \\ 
&>& \frac12.
\end{eqnarray*}
On the other side, $\zeta_0 \leq \frac{H^T}{H^T+1} < 1.$ Hence, $0 < F(x_0)\leq \frac{H}{2}.$

Let $0 \leq x_0 < 1,$ $p\in \{1,\ldots,\,T\}$ is the smallest number at which $f^{(p-1)}(x_0) < 1,$ but $f^{(p)}(x_0)>1$
(here we mean that $f^{(0)}(x):=x$). It is clear that $f^{(p)}(x_0) \leq \frac{H}{2}.$

Note that
\leqnomode
\begin{align}\label{mis9}
\frac{1}{H} f^{(p)}(x_0) &\leq f^{(p-1)}(x_0) \leq 1 - \frac{1}{H} f^{(p)}, \nonumber \\
\frac{1}{H^2} f^{(p)}(x_0) &\leq \frac{1}{H} f^{(p-1)}(x_0) \leq f^{(p-2)}(x_0) \nonumber \\
&\leq  1-\frac{1}{H} f^{(p-1)}(x_0) \leq 
1 - \frac{1}{H^2} f^{(p)}(x_0), \\
\quad \quad \quad & \quad \quad \vdots\quad \quad \quad , \nonumber \\
\frac{1}{H^p} f^{(p)}(x_0) &\leq x_0 \leq 1 - \frac{1}{H^p} f^{(p)}(x_0) .
\nonumber
\end{align}

If $p=T,$ then $\zeta_0  > 0$ and, since $x_0 - f^{(T)}(x_0) < 0,$ 
$$
\zeta_0=\vartheta x_0 + (1-\vartheta)f^{(T)}(x_0) < 
\frac{1}{H^T + 1} \left ( \left(H^T - \frac{1}{H}\right)x_0 + \left(1 + \frac{1}{H}\right) f^{(T)}(x_0) \right),
$$
and because of (\ref{mis9}), 
\begin{align*}
\left(H^T - \frac{1}{H}\right)x_0 &+ \left(1 + \frac{1}{H}\right) f^{(T)}(x_0) \\
&\leq \left(H^T - \frac{1}{H}\right)
\left(1 - H^{-T} f^{(T)}(x_0)\right) + 
\left(1 + \frac{1}{H}\right) f^{(T)}(x_0) \\
&= H^T - H^{-1} -f^{(T)}(x_0) + H^{-T-1}f^{(T)}(x_0) + f^{(T)}(x_0) +
H^{-1}f^{(T)}(x_0) \\
&< H^T + \frac12 H^{-T} +\frac12 \\
&< H^T +1.
\end{align*}
This implies $\zeta_0 < 1.$

Let $p<T.$ Then $\zeta_0 < 1.$ Check the inequality $\zeta_0 > 0.$ Since 
$f^{(T)}(x_0) = H^{T-p} \left(1 - f^{(p)}(x_0)\right) \leq 0,$ then $x_0 - f^{(T)}(x_0) > 0$ and
 $$
\zeta_0=\vartheta x_0 + (1-\vartheta)f^{(T)}(x_0) > 
\frac{1}{H^T + 1} \left ( \left(H^T - \frac{1}{H}\right)x_0 + \left(1 + \frac{1}{H}\right) f^{(T)}(x_0) \right).
$$
Because of (\ref{mis9}), 
\begin{eqnarray*}
\left(H^T - \frac{1}{H}\right)x_0 &+& \left(1 + \frac{1}{H}\right) f^{(T)}(x_0)  \\
&\geq& 
\left(H^T - \frac{1}{H}\right) H^{-p} f^{(p)}(x_0) +
    \left(1 + \frac{1}{H}\right) H^{T-p} \left(1- f^{(p)}(x_0)\right) \\
&=& H^{T-p} f^{(p)}(x_0) - H^{T-p-1} f^{(p)}(x_0)
+ \left(1 + \frac{1}{H}\right) H^{T-p} \\
&&  - \left(H^{T-p}+H^{T-p-1}\right)f^{(p)}(x_0)\\
&=&\left(1 + \frac{1}{H}\right) H^{T-p}  
- 2H^{T-p-1}f^{(p)}(x_0) \\
&\geq& \left(1 + \frac{1}{H}\right) H^{T-p} - H^{T-p} \\
&>& 0.
\end{eqnarray*}
Thus, when $0 \leq x_0 < 1,$ the inequalities $0<\zeta_0 <1$ hold and, therefore, $0 < F(x_0) \leq \frac{H}{2}.$ 
The lemma is proved. 
\end{proof}

 Theorem~\ref{t3} follows from  Lemmas~\ref{l5} and \ref{l6}.

Note that when the inequalities $\displaystyle{ \frac{H^T}{H^T+1}<\vartheta<\frac{H^T+\frac{1}{H}}{H^T+1}}$ are satisfied,
the invariant set of equation~(\ref{ceq}) will no longer be a segment, but will be the union of a finite or countable 
number of intervals, and the measure of this set may be small.

Theorems \ref{t2} and \ref{t3} are illustrated in Section \ref{sec:GraphicalStudy} for the case $H = 3$ and $T = 2$.
\section{Computational particularities of using the control system} \label{sec:ComputationalParticularities}

 In this section we introduce the residuals: 
 $$U_n = ||f(\vartheta x_n + (1 - \vartheta) f^{(T)}(x_n)) - f(x_n)||$$ and $$\hat{U}_n = || x_{n+T} - x_n||,$$ whose rate of decay allows us to compare solutions of system (5) to solutions of system (1), as well as the $T$-periodicity of the solutions. 

Let us consider the computational particularities of the iterative scheme~(\ref{ceq}) for finding the cycles of equation~(\ref{eq}).
For cycles with positive multipliers, it is theoretically possible to take any number from the interval 
$\left(\frac{H^T - \frac{1}{H}}{H^T-1},\,\frac{H^T + \frac{1}{H}}{H^T-1}\right)$ as the control parameter;
however, if this parameter belongs to the half-interval $\left(\frac{H^T - \frac{1}{H}}{H^T-1},\,\frac{H^T}{H^T-1}\right],$
then the fixed point will have a fairly large basin of attraction. This will be illustrated geometrically below. 
Therefore, if we want to find true period-$T$ cycles, it is reasonable to choose the control parameter closer to $\displaystyle{\frac{H^T + \frac{1}{H}}{H^T-1}}.$ For cycles with negative 
multipliers, the control parameter can be taken from the interval 
$\displaystyle{\left(\frac{H^T - \frac{1}{H}}{H^T + 1},\,\frac{H^T}{H^T + 1}\right)}.$ Since the basin of attraction of a given cycle can be quite small, the initial value $x_0$ should belong to the nodes of a sufficiently dense grid of the interval $(0,\,1)$ in order to find the largest possible 
number of cycles.

Note that for large values of $T,$ the control parameter $\vartheta$ is close to one, and the value $1-\vartheta$ is close to zero. 
In this case, the lengths of the intervals of possible changes in the control parameter are equal to $\frac{1}{H(H^T-1)}$ or
$\frac{1}{H(H^T+1)},$ i.e. as $T$ grows, they tend to zero exponentially. 

Therefore, although the method suggested above for determining cycles theoretically allows us to solve 
the stated problem numerically, practical questions remain: when can we rely on numerical solutions? How can we control numerical results? 
What calculation accuracy should be chosen?

In practice, intermediate calculations should be introduced to control the results. Let the sequence $\left\{x_n\right\}_{n=1}^{\infty}$ 
be an orbit of the system given by Equation~(\ref{ceq}), and let the quantity $1-\vartheta$ have the order of magnitude $10^{-p},$ where 
$p$ is large enough. 
Then the first checkpoint will be the estimate of the residual 
$U_n = \left\Vert f\left( \vartheta x_n + (1 - \vartheta) f^{(T)}(x_n)\right) - f(x_n)\right\Vert.$ If the sequence $\left\{x_n\right\}$ 
tends to an orbit of system~\eqref{eq}, then the sequence $\left\{U_n\right\}$ tends to zero. However, if the sequence 
$\left\{x_n\right\}$ does not tend to an orbit of  system~\eqref{eq}, then the residual can have the order of magnitude 
${1 - \vartheta}\sim 10^{-p},$ i.e. be close to zero. To be sure that the residual tends to zero, we have to choose the calculation 
accuracy $\delta=10^{-p_1},$ where $p_1$ should be significantly greater than $p.$ Then the first point of control will be the condition
${U_n}\sim 10^{-p_1},$ $n\geq n_1.$

The second checkpoint is the check of periodicity of the numerical solution: 
$\hat{U}_n = \left\Vert x_{n+T} - x_n\right\Vert \sim 10^{-p_1},$ $n\geq n_1.$ Of course, it is also necessary to check that $T$ is a proper cycle of system \eqref{eq},
and not a subcycle, i.e. a cycle of shorter length.

The checkpoints give necessary conditions for the sequence $\left\{x_n\right\}$ to be a $T$-cycle 
of equation~(\ref{eq}). The effectiveness of these necessary conditions is that they are quite simple to check, which we illustrate in the following example.

\textsc{Example.} Consider the problem of finding 5-cycles of equation~(\ref{eq}) when $H=4.$ Choose two initial conditions and for each of them find one 
5-cycle with a positive and negative multiplier respectively. Set $x_0 \in \{0.25,\,0.85\},$ 
$\vartheta \in \left\{\frac{H^T - \frac{0.4}{H}}{H^T+1},\,\frac{H^T + \frac{0.4}{H}}{H^T-1} \right\}.$ Thus, we got four iterative 
schemes. Choose one color for each scheme in order to visualize the results:

\begin{center}
\begin{tabular}{c || c | c| } 
$x_0$/$\vartheta$ &  $\displaystyle{ \frac{H^T - \frac{0.4}{H}}{H^T+1}\approx 0.9989}$ & $\displaystyle{\frac{H^T + \frac{0.4}{H}}{H^T-1}\approx 1.0010}$ \\ \hline \hline
$0.25$ & red & green \\ \hline
$0.85$ & blue & black \\ \hline 
\end{tabular}
\end{center}

Since $(1 - \vartheta)$ is on the order of $10^{-3}$, we choose a  calculation accuracy of $10^{-15}.$ The corresponding cyclic points are shown in Figure~\ref{f1}. Figure~\ref{f5} shows 
the graphs of the residuals $U_n$ and $\hat{U}_n$ as $n$ increases. The periodicity condition, $\hat{U}_ = \Vert x_{n+5} - x_n\Vert < 10^{-15}$ holds for all four schemes, 
starting from $n=43.$

We note that there are 6 distinct cycles of period 5, but we only show four of these here.  To find the remaining two cycles, we would judiciously select two more values of $x_0$ and $\vartheta$.

All cycles of any length can be found in a similar fashion.  The limitation is the calculation accuracy, which should be chosen to be approximately $H^{1.05 T}$, since $H = |f'(x)|$. 
Figure~\ref{f6} (Left Panel) shows the cyclic points of four cycles of length 100. The accuracy of calculations was taken $10^{65}.$ For negative multipliers, 
the control parameter is chosen to be $\vartheta = \displaystyle{\frac{H^T - \frac{0.4}{H}}{H^T+1}} \approx 1 - 1.78 \cdot 10^{-61};$ for positive multipliers, 
$\vartheta = \displaystyle{ \frac{H^T + \frac{0.4}{H}}{H^T-1}} \approx  1 + 0.68 \cdot10^{-62}.$ Note that the lengths of the intervals of possible changes 
in the control parameter have the order of magnitude $H^{-(T+1)} \approx 1.5 \cdot 10^{-61}$. 
The necessary bounds on the residuals are attained at the 250th step (Figure~\ref{f6}, Right Panel).

\begin{figure}[h!]
\centering
\includegraphics[scale=0.28]{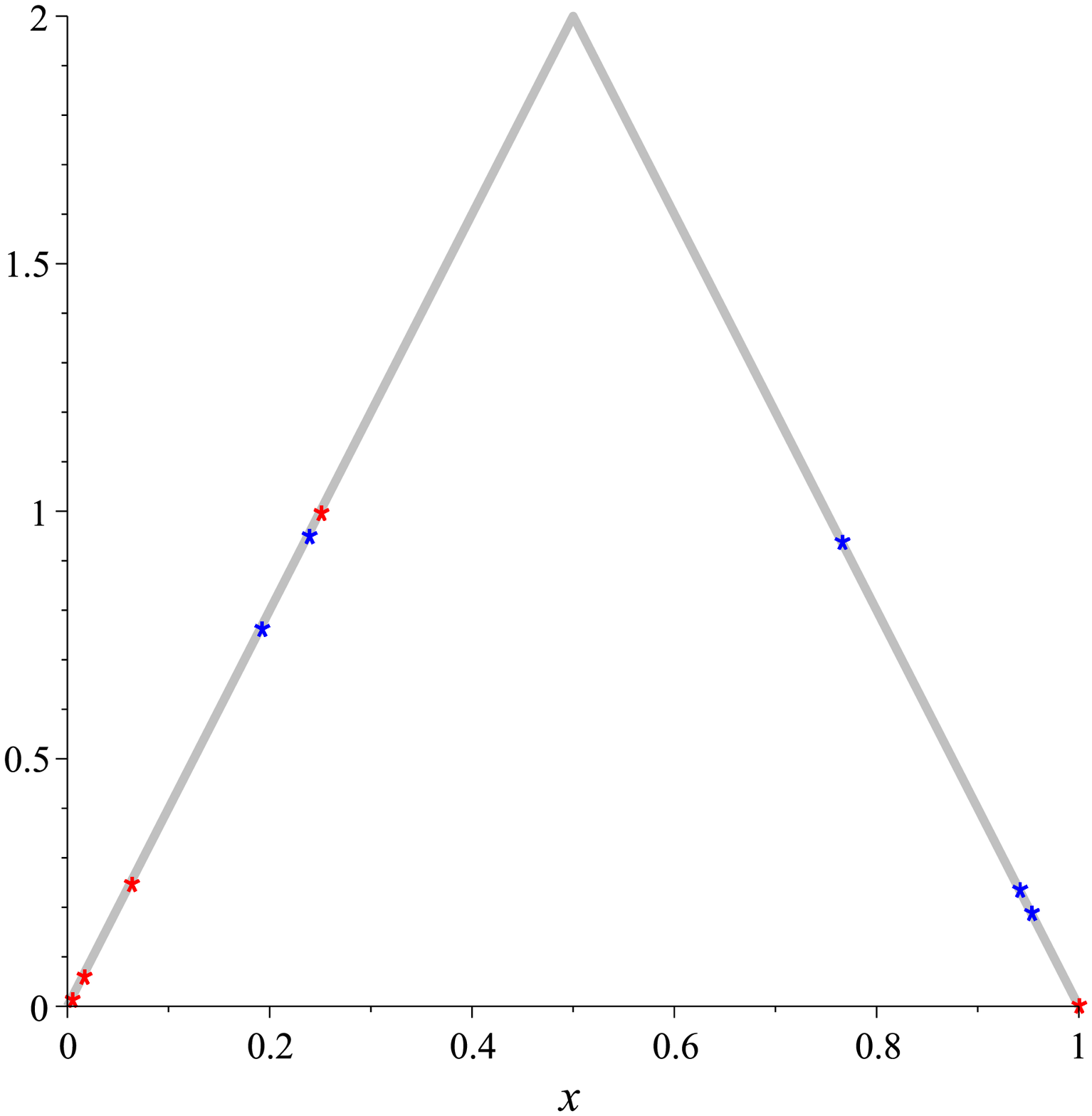}
\hspace{1cm}
\includegraphics[scale=0.28]{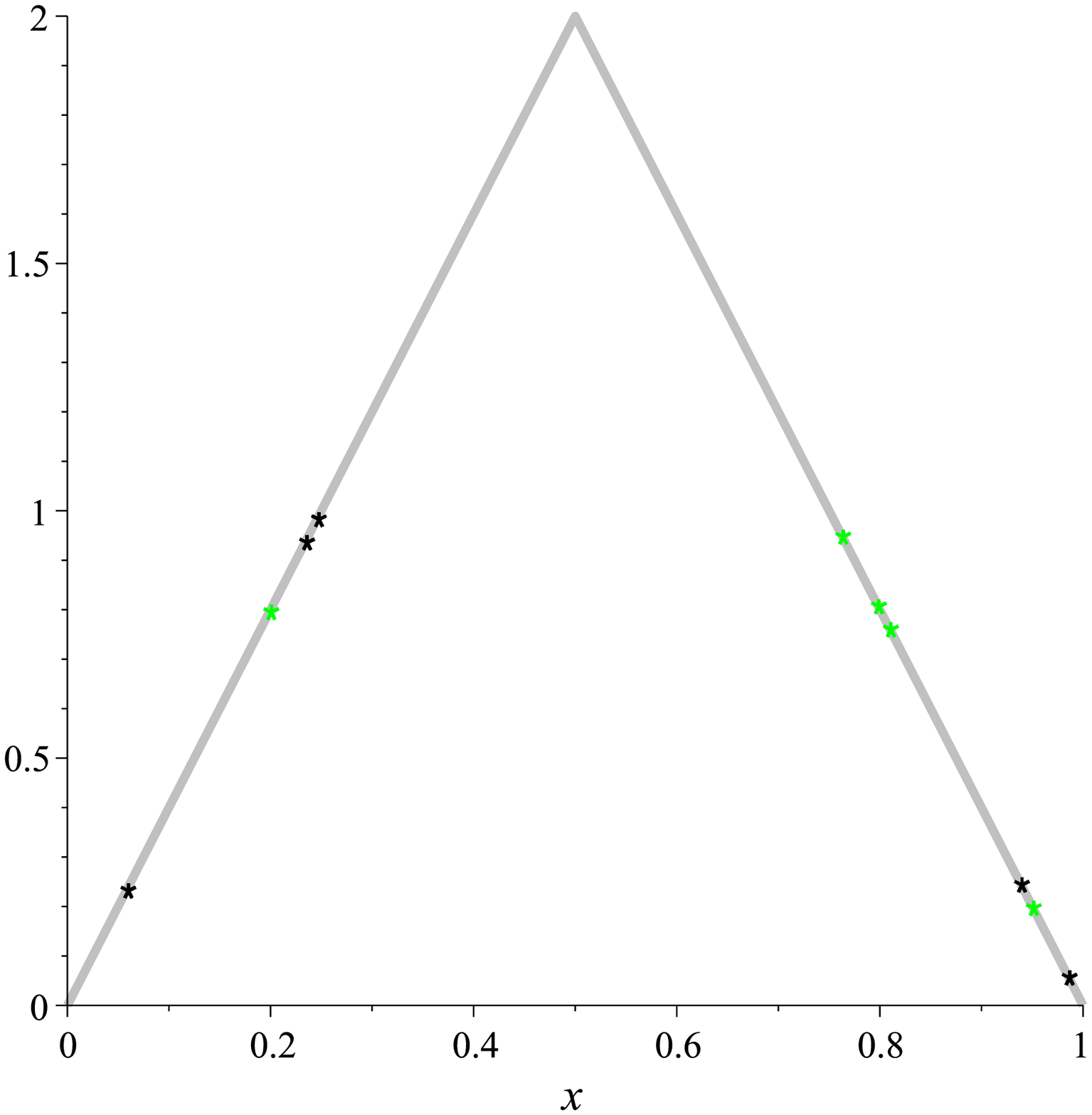}
\hspace{1cm}
\includegraphics[scale=0.28]{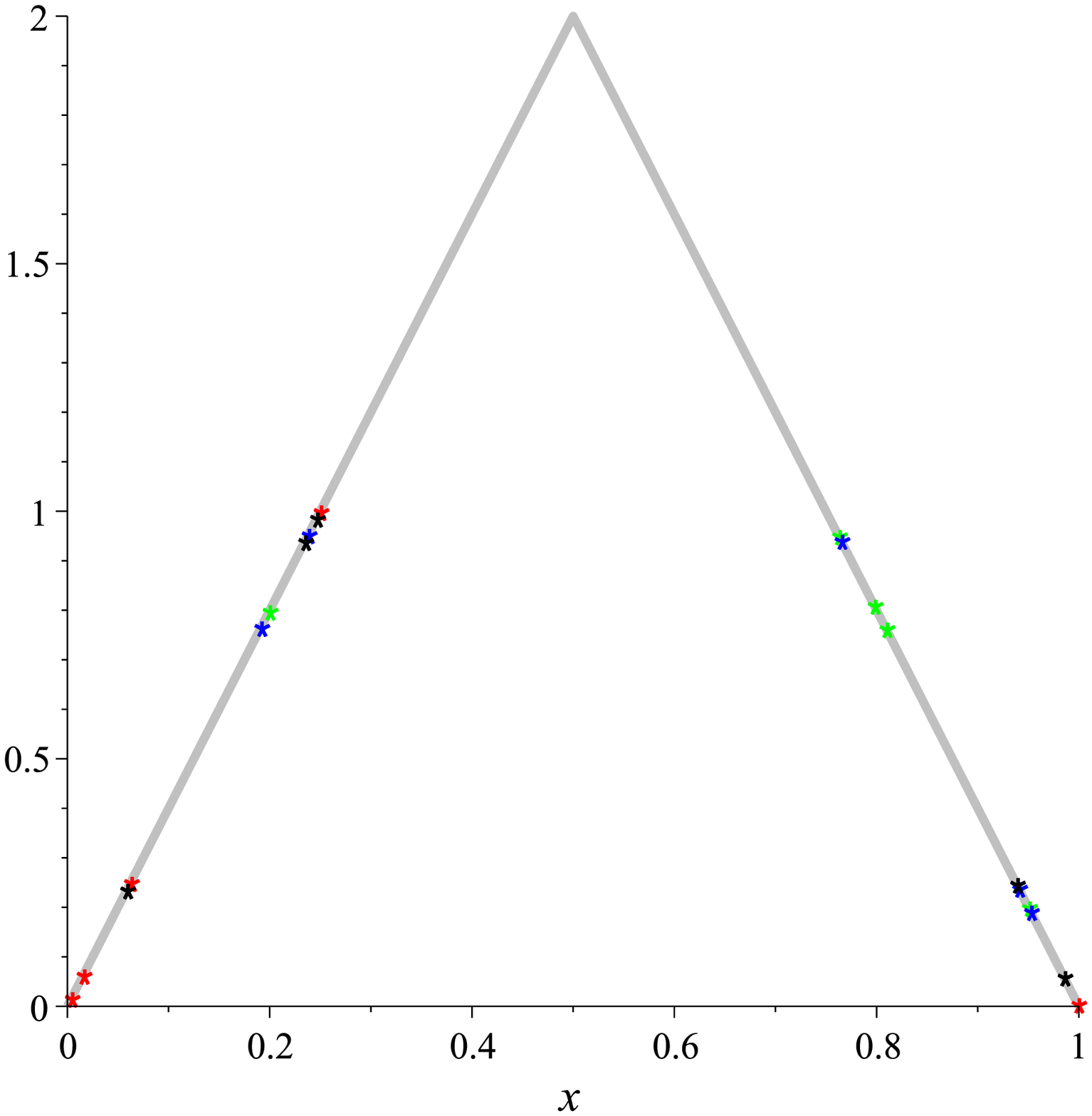}
\caption{Four 5-cycles of the tent map.  The upper left panel shows two cycles with positive multipliers (in red and blue), and the upper right panel shows two cycles with negative multipliers (in green and black). The lower graph shows all four cycles superimposed on each other, showing how the basins of attraction are intertwined.} \label{f1}
\end{figure}

\begin{figure}[h!]
\begin{tikzpicture}
\node[label = \hbox{$U_n, \vartheta = 0.99893$}] (Un) {\includegraphics[scale=0.28]{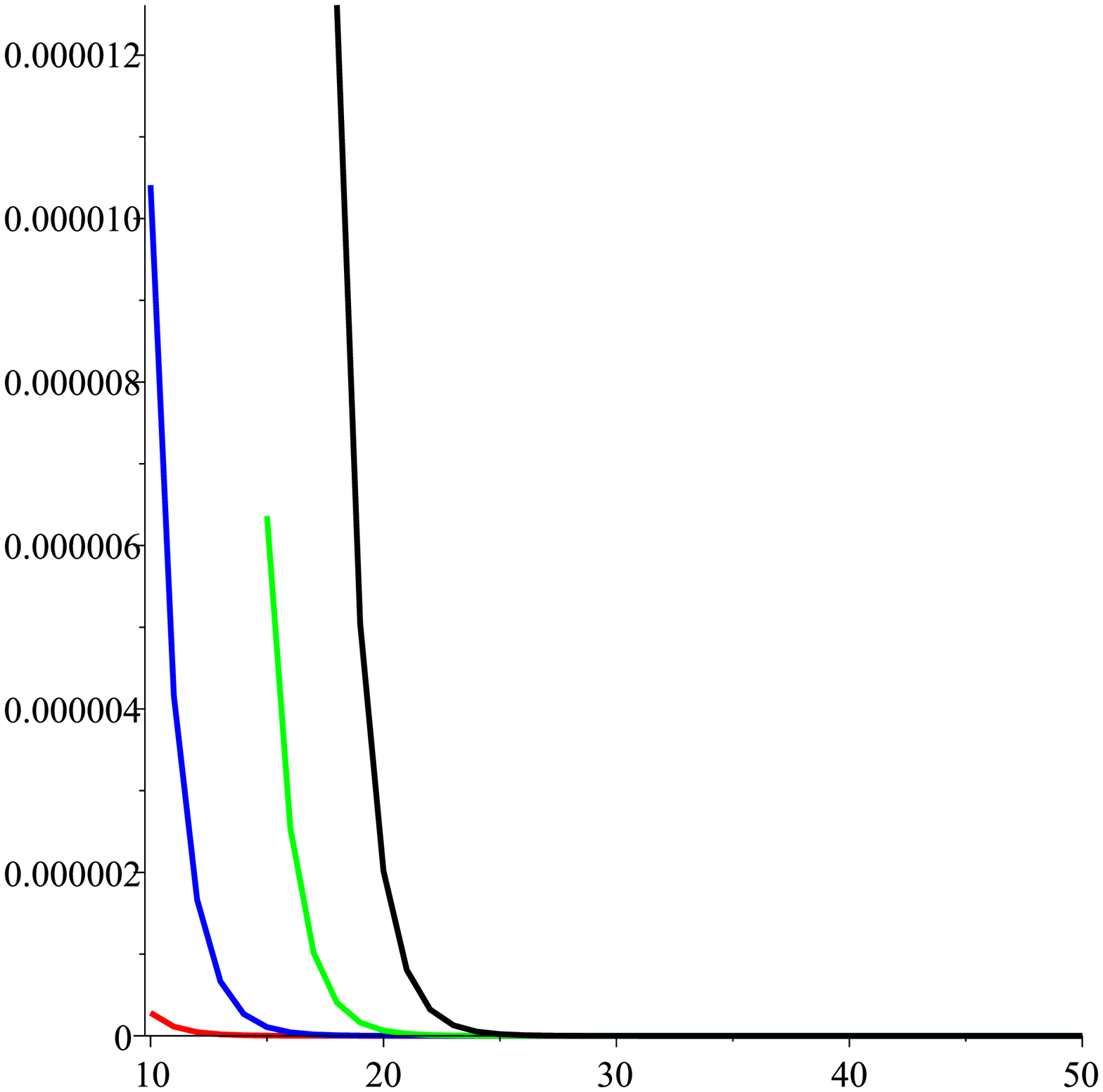}};
\node[right of = Un, label = 
\hbox{${U}_n, \vartheta = 1.0010$}, node distance = 7cm] (hatUn)
{\includegraphics[scale=0.28]{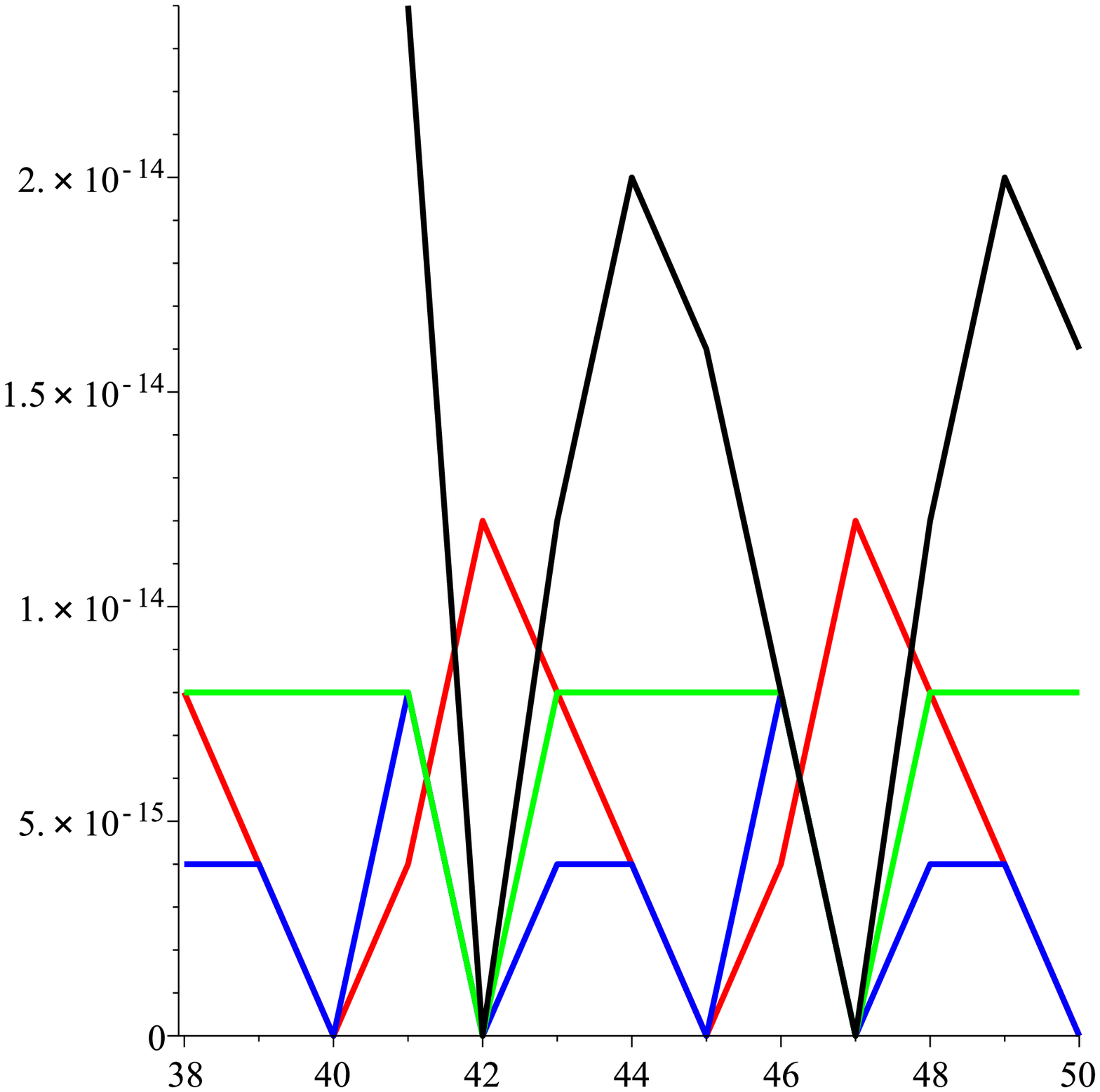}};
\end{tikzpicture}
\caption{Residuals for the numerically estimated period-5 points.  Left Panel: residuals $U_n$ where $n$ increases from 10 to 50, with $\vartheta \approx 0.99893$; Right Panel: the residuals ${U}_n$ with $\vartheta \approx 1.0010$ for $n$ in the range $38, \dots 50$. The control point condition is satisfied for $n \ge 43$.} \label{f5}
\end{figure}

\begin{figure}[h!]
\centering
\includegraphics[scale=0.28]{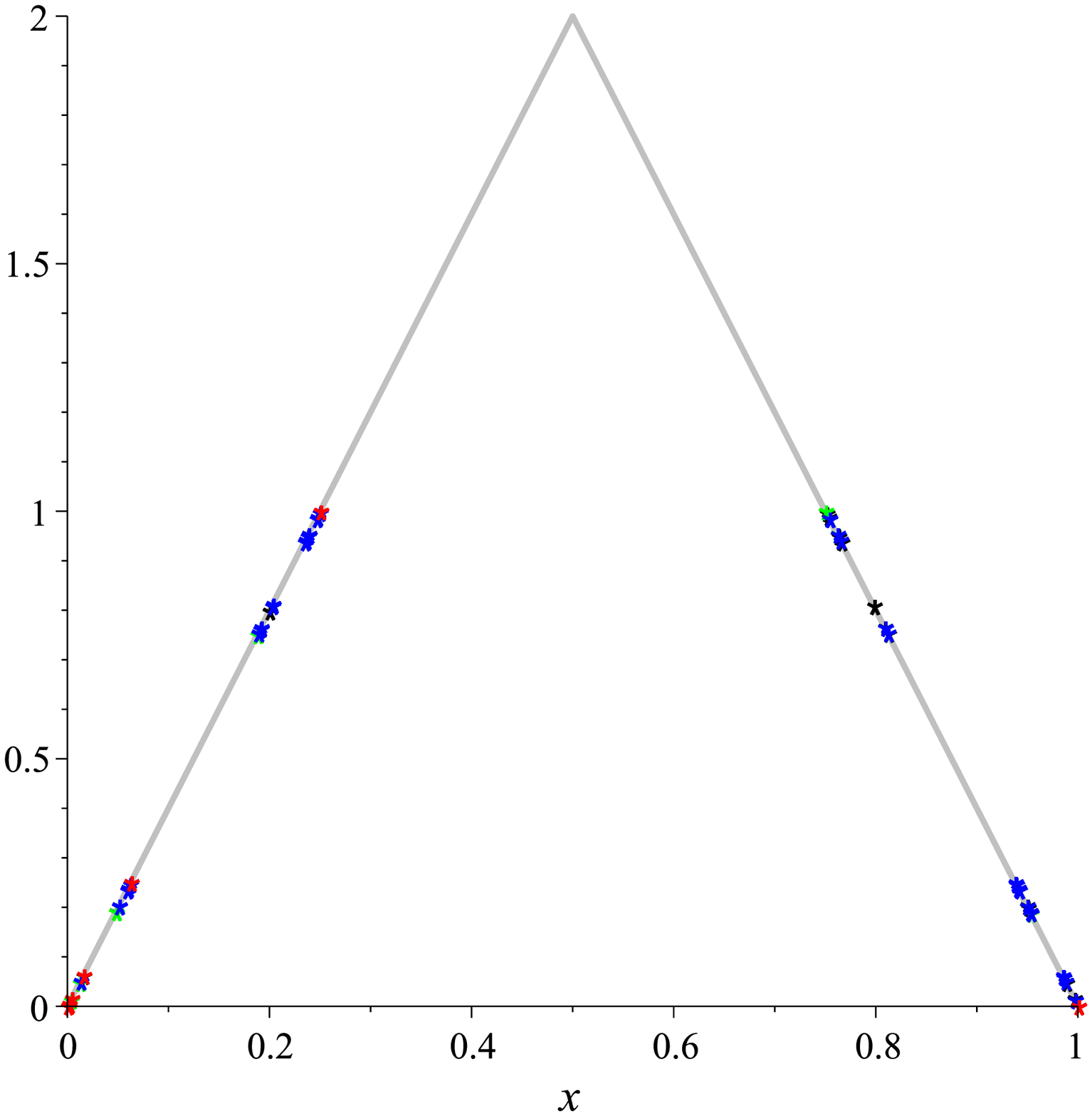}
\includegraphics[scale=0.28]{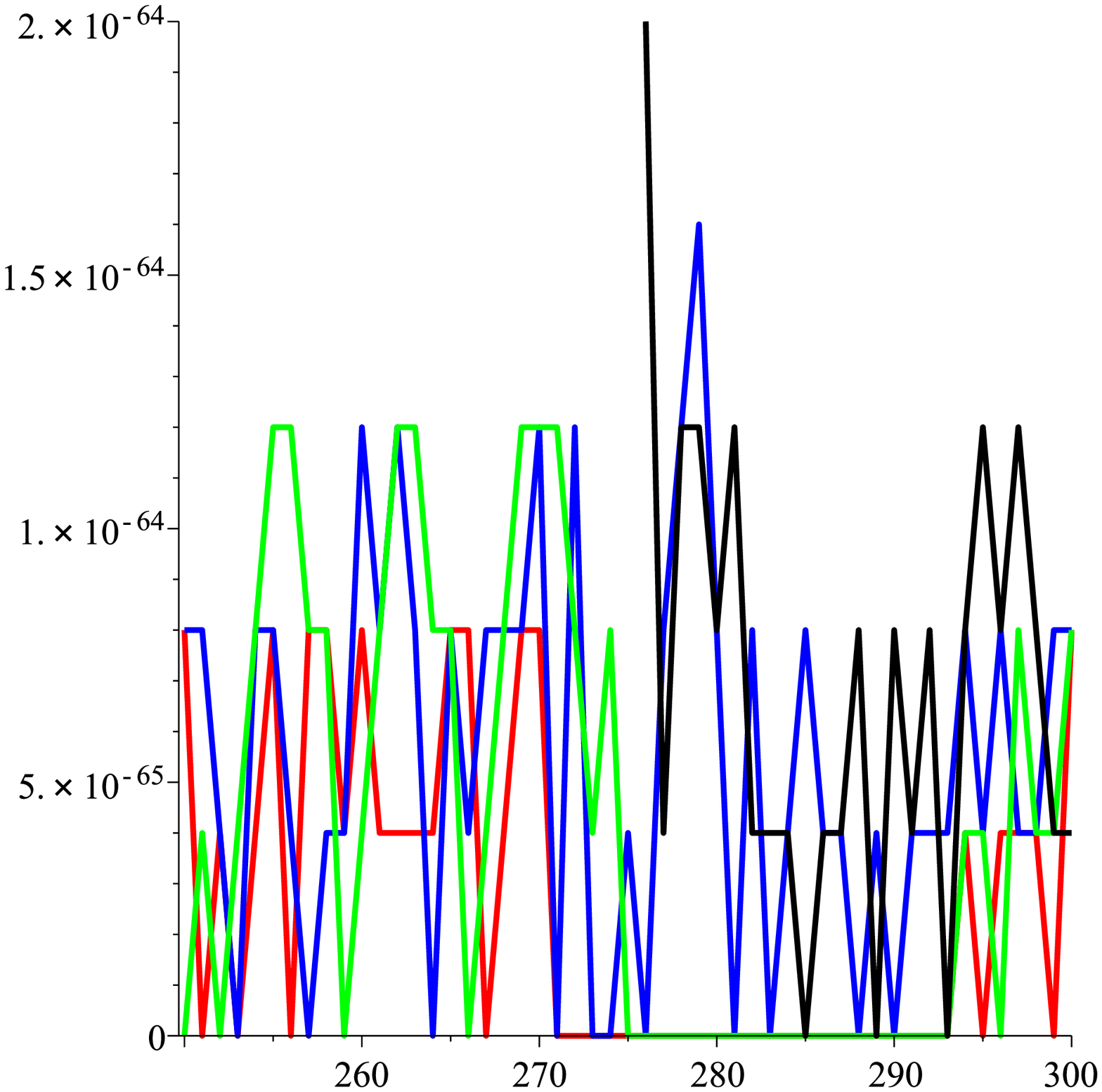}
\caption{Four sets of cyclic points of the tent map 100-cycles.  Left Panel: the four cycles are colored red, blue, green and black.  Notice that many points in the four cycles are tightly clustered.  Right Panel: the residuals, $U_n$, for each of the four cycles as a function of $n$, where $n$ increases from 250 to 300.  The control point conditions are satisfied for $n \ge 250$ in all four cases.  } \label{f6}
\end{figure}

\section{Graphical study of the control system map} \label{sec:GraphicalStudy}

For some choices of the control parameter $\vartheta$ the invariant set is visible on the graph of $y = f \left( \vartheta x + (1 - \vartheta)f(x) \right).$ Several choice of $T$ and $\vartheta$ are illustrated below.

Equation~(\ref{eq}) has two fixed points $\eta = 0$ and $\eta = \frac{H}{H+1}$,  both of which are unstable since
their multipliers are $H$ and $-H$ respectively. Consider Equation~(\ref{ceq})
$$
x_{n+1} = f \left( \vartheta x_n + (1 - \vartheta)f(x_n) \right)
$$
for different values of $\vartheta \in \left[0,\, \frac{H^T + \frac{1}{H}}{H^T-1}\right].$ For definiteness, set $H=3.$
Then $\left[0,\, \frac{H^T + \frac{1}{H}}{H^T-1}\right] = \left[ 0,\,\frac53\right].$ According to Theorem \ref{t2},  to stabilize the fixed point at  $\eta=0,$
we need to choose the control parameter from the interval $\left(\frac43,\,\frac53\right)$.  By Proposition \ref{Prop1}  the multiplier of this fixed point is $\lambda = 9 - 6 \vartheta$, which decreases from 1 to -1 as the control parameter $\vartheta$ increases from $\frac{9}{3}$ to $\frac{5}{3}$, and $\lambda = 0$ when $\vartheta = \frac{3}{2}$.   The graphs of the function 
$F(x) = f \left( \vartheta x + (1 - \vartheta)f(x) \right)$ for different values of $\vartheta$ are shown in Fig.~\ref{f8}.

\begin{figure}[h!]
\centering
\includegraphics[scale=0.28]{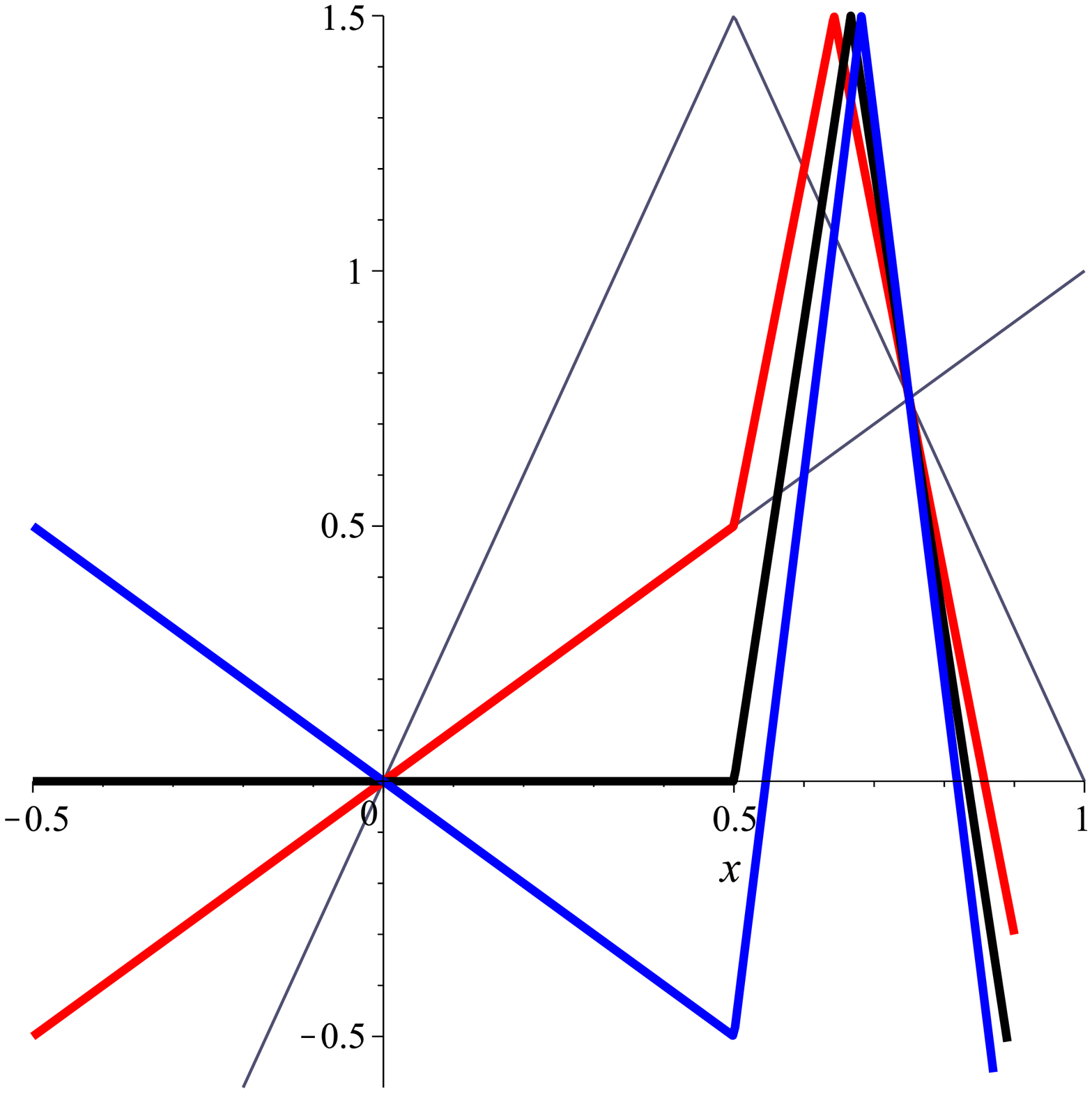}
\caption{Graphs of the function $y = f \left( \vartheta x + (1 - \vartheta)f(x) \right)$ at $\vartheta=\frac43$ (red);
at $\vartheta=\frac32$ (black); at $\vartheta=\frac53$ (blue); graphs of the functions $y=x$ and $y=f(x)$ 
are marked in grey} \label{f8}
\end{figure}

According to Theorem \ref{t3}, the fixed point at  $\eta=\frac34$ will be a locally asymptotically stable fixed point of equation~(\ref{ceq}) if
$\vartheta \in \left(\frac23,\,\frac56\right).$ As the control parameter increases, the multiplier of this equilibrium 
decreases from $1$ to $-1.$ When $\vartheta=\frac34,$ the multiplier equals zero. The graphs of the function 
$F(x) = f \left( \vartheta x + (1 - \vartheta)f(x) \right)$ for different values of $\vartheta$ are depicted in Fig.~\ref{f9}. 

\begin{figure}[h!]
\centering
\includegraphics[scale=0.28]{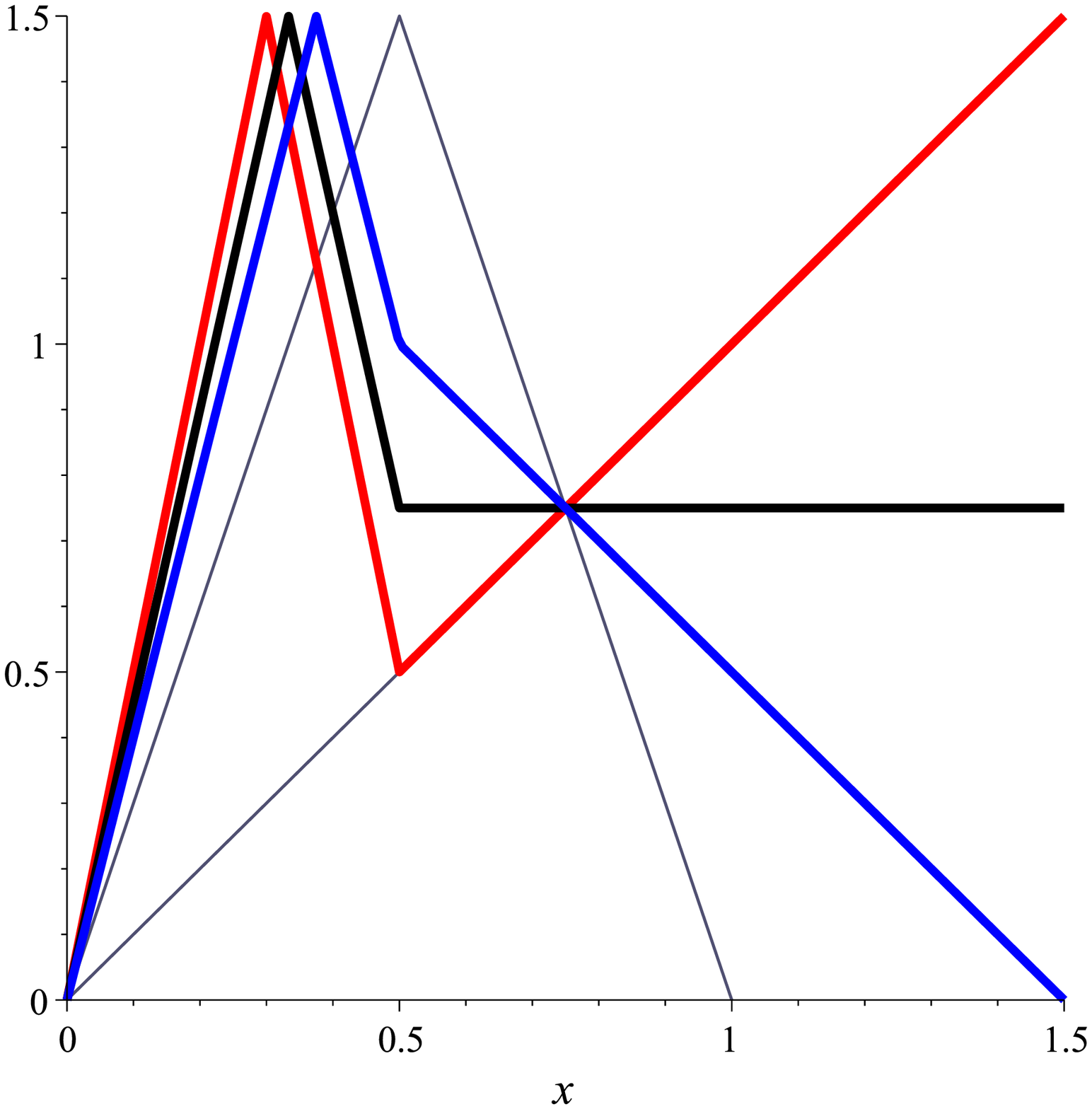}
\caption{Graphs of the function $y = f \left( \vartheta x + (1 - \vartheta)f(x) \right)$ at $\vartheta=\frac23$ (red);
at $\vartheta=\frac34$ (black); at $\vartheta=\frac56$ (blue); graphs of the functions $y=x$ and $y=f(x)$ 
are marked in grey} \label{f9}
\end{figure}

If we represent the Lamerey diagram  on the graphs of Figure~\ref{f8}, then we can observe that for any initial value $x_0$ 
and $\vartheta \in \left(\frac43,\,\frac53\right),$ the corresponding solution of equation~(\ref{ceq}) will tend to zero. 
Similarly, Figure~\ref{f9} shows that the segment $\left[ 0,\,\frac32\right],$ is mapped to itself under
$F(x) = f \left( \vartheta x + (1 - \vartheta)f(x) \right)$,  and for $\vartheta \in \left(\frac23,\,\frac56\right)$) it is mapped strictly {\it into} itself.
Moreover, for any $x_0 \in \left(0,\,\frac32\right),$ the solution tends to the fixed point $\eta = \frac34.$ These facts follow from the proofs of Theorems \ref{t2} and \ref{t3}.

Let $T=2.$ Equation~(\ref{eq}) has only one 2-cycle, and its multiplier is negative. In this case, Equation~(\ref{ceq}) 
takes the form 
$$x_{n+1} = F(x) = f \left( \vartheta x_n + (1 - \vartheta)f^{(2)}(x_n) \right).$$
Consider the graphs of the functions 
$y=F(x)$, and  $y=F^{(2)}(x)$, when
$\vartheta = \displaystyle{\frac{H^T}{H^T-1} = \frac{9}{8}}$ and $\vartheta = \displaystyle{\frac{H^T}{H^T+1} = \frac{9}{10}}$,
shown in Figures~\ref{f10} and \ref{f11}, respectively.

\begin{figure}[h!]
\centering
\includegraphics[scale=0.28]{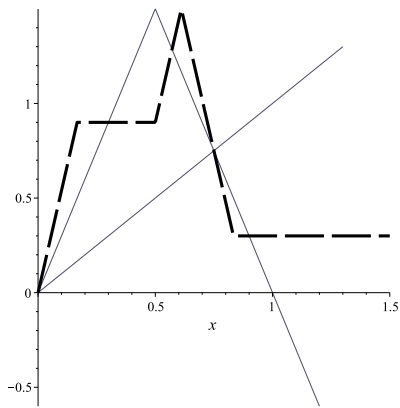}
\caption{Graphs of the function $y = F(x)$ at $\vartheta=\frac{9}{10}$ (black dashed line);
$y=x$ and $y=f(x)$ (grey).
Both fixed points are unstable, while $F'(x) = 0 $ at both points on the period-2 cycle: $\{ \frac{3}{10}, \frac{9}{10} \}$, so the period-2 cycle is super-stable. } \label{f10}
\end{figure}

\begin{figure}[h!]
\begin{minipage}[h]{0.45\linewidth}
\center{\includegraphics[scale=0.28]{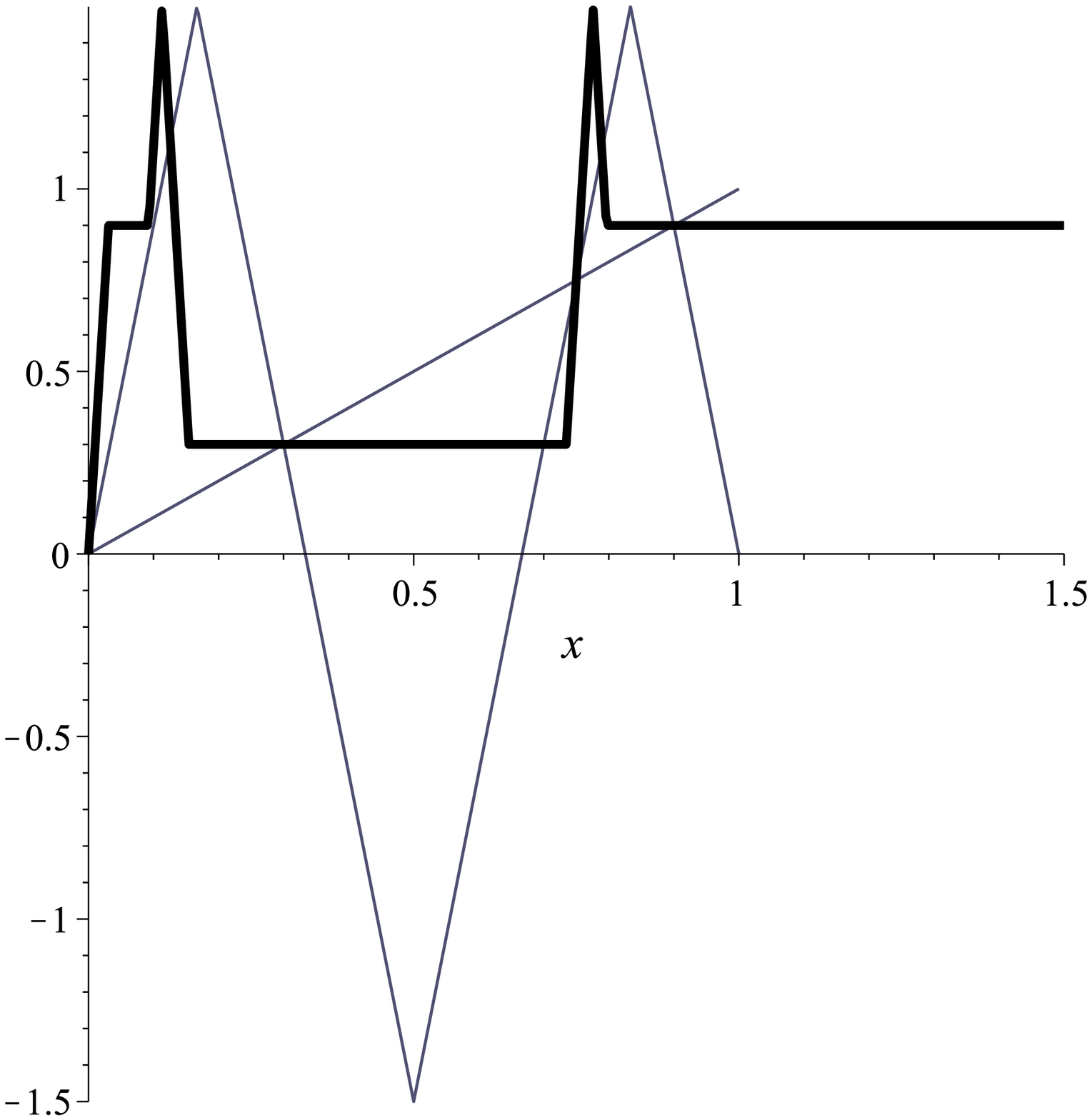} \\ a)}
\end{minipage}
\hspace{1cm}
\begin{minipage}[h]{0.45\linewidth}
\center{\includegraphics[scale=0.28]{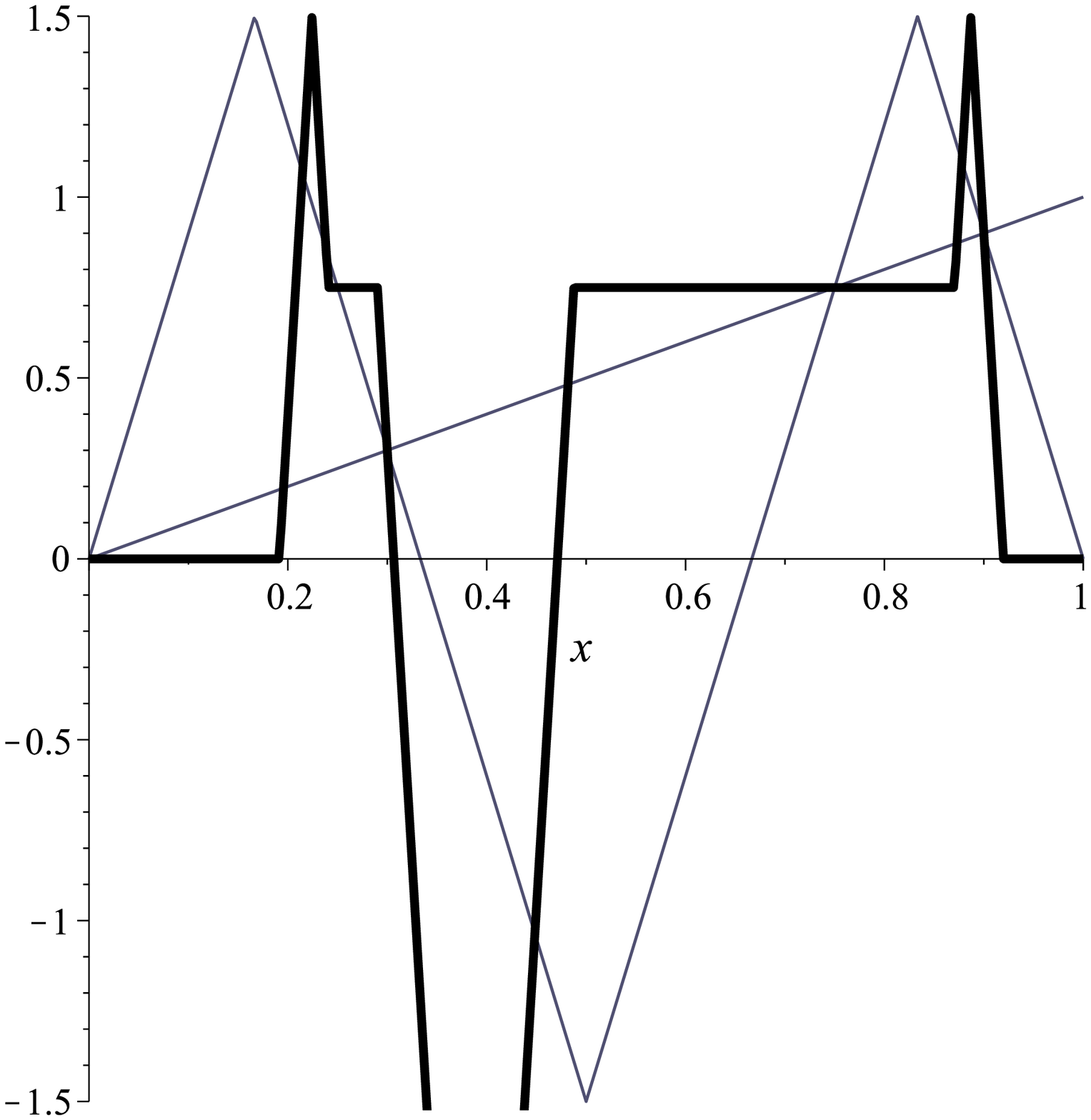} \\ b)}
\end{minipage}
\caption{Graphs of the function $y = F^{(2)}(x)$ are drawn in black.  In the Left Panel a): $\vartheta=\frac{9}{10}$; in the Right Panel b): $\vartheta=\frac{9}{8}$.
Graphs of the functions $y=x$ and $y=f^{(2)}(x)$ are shown in grey.
In a) we see that the derivative of $^{(2)}(x)$ is 0 at the two fixed points of $F^{(2)}$ that are in the proper period-2 cycle of $f$, showing that these points are {\it super-stable} fixed points of $F^{(2)}$.  By contrast, in b) the period-2 points of $f$ are now unstable, while the two fixed points of $f$ are super-stable.}\label{f11}
\end{figure}

From Figure~\ref{f10}, we can see that the set $\left[ 0,\,\frac32\right]$ is invariant under the mapping $y=F(x)$.  In  Figure~\ref{f11}-a), we see that 
 the 2-cycle of Equation~(\ref{ceq}) becomes locally asymptotically stable with the multiplier equal to zero. In
Figure~\ref{f11}-b), we see that the 2-cycle of equation~(\ref{ceq}) is unstable, but both fixed points are locally asymptotically stable,  
with zero multipliers. We give an explanation for this fact in the next few paragraphs.

Consider the global behavior of solutions of equation~(\ref{ceq}) at $\vartheta = \frac{H^T}{H^T+1}.$ The function 
$$
\zeta (x) = \frac{H^T}{H^T+1} x + \frac{1}{H^T+1}	f^{(T)}(x)
$$
does not decrease on $(-\infty,\infty).$ Indeed, 
$$
\zeta'(x) = \frac{H^T \pm H^T}{H^T+1} =
\left\{\begin{array}{ll}
0,\,x\in \Sigma, \\
\frac{2 H^T}{H^T+1},\,x \notin\Sigma,
\end{array}\right.
$$
where $\Sigma$ is the set on which the function $f^{(T)}(x)$ decreases. Note that the sets on which  $f^{(T)}(x)$ 
and $\zeta (x)$ are increasing coincide. In particular, the function $\zeta (x)$ increases when 
$x \in \left( \frac12,\, 1-\frac{1}{H}+\frac{1}{2 H^{T-1}} \right)$ from $-\frac12 H^{T-1} (H-1)$ to $\frac{H}{2}.$ The point 
$\left( \frac12,\, -\frac12 H^{T-1} (H-1) \right)$ is the minimum of the function $f^{(T)}(x)$ on $[0,1]$. Figure~\ref{f12} 
shows the graphs of the functions $f^{(T)}(x)$ and $\zeta (x)$ in the case $H=3,$ $T=3.$

\begin{figure}[h!]
\centering
\includegraphics[scale=0.28]{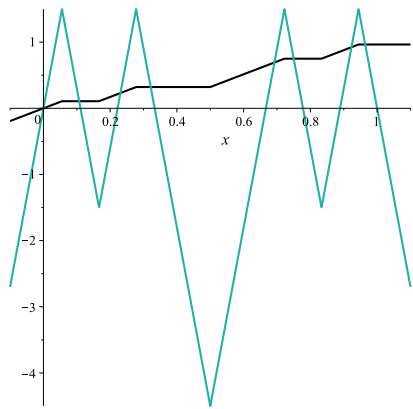}
\caption{Graphs of the functions $f^{(T)}(x)$ (sea green) and $\zeta (x)$ (black) when  $H=3$ and  $T=3$.  
Note that $\zeta(x)$ and $f^{(T)}$ are increasing on the same intervals, and $\zeta(x)$ is constant where $f^{(T)}(x)$ is decreasing.} \label{f12}
\end{figure}

Let $x \in \left( \frac12,\, 1-\frac{1}{H}+\frac{1}{2 H^{T-1}}  \right),$ $T>1.$ Denote $x = 1 - \frac{1}{H} + \frac{\alpha}{H^T}.$
Then $\alpha \in \left( -\frac12 H^{T-1} (H-2),\, \frac{H}{2} \right).$ Compute $f^{(T)}(x)$: the first iterate remains greater than 1/2, the second iterate is less than 1/2, and all subsequent iterates increase to $\alpha$: 
$$f(x) = 1 - \frac{\alpha}{H^{T-1}}>\frac12, \;
f^{(2)}(x) = \frac{\alpha}{H^{T-2}}<\frac12, \ldots , f^{(T)}(x)=\alpha.$$
This implies that the equation 
$\frac{H^T}{H^T+1} \left( 1-\frac{1}{H}+\frac{\alpha}{H^T} \right) + \frac{1}{H^T+1} \alpha = \frac12$ 
has only one root on $\left( -\frac12 H^{T-1} (H-2),\, \frac{H}{2} \right)$, which we denote by  $\widehat{\alpha} = \frac14 - \frac14 H^T + \frac12 H^{T-1}.$
Since $\zeta \left(\frac12\right) = \frac{H^{T-1}}{H^T+1} < \frac12$, and  
$\zeta \left( 1 - \frac{1}{H} + \frac{1}{2 H^{T-1}} \right) = \frac{H^T - H^{T-1} + H}{H^T +1} > \frac12,$ then in the interval 
$x \in \left( \frac12,\, 1-\frac{1}{H}+\frac{1}{2 H^{T-1}}  \right)$ there is exactly one root of the equation $\zeta (x) = \frac12$, namely
$\hat{x} = 1 - \frac{1}{H} + \frac{\widehat{\alpha}}{H^T} = \frac34 - \frac{1}{2H} + \frac{1}{4 H^T}.$ This means that the function 
$F(x)=f(\zeta(x))$ does not decrease on the interval $\left[0,\,\hat{x}\right],$ does not increase on $\left[\hat{x},\,1\right],$
$F(\hat{x})=\frac{H}{2}$ and $F(x) = F \left( 1 - \frac{1}{2 H^{T-1}} \right) = \frac{H^T}{H^T + 1}$ when 
$x \geq 1 - \frac{1}{2 H^{T-1}}.$ The graph of the function $F(x)$ when $H=3$ and $T=4$ is shown in Figure~\ref{f13}. On the same figure we plot the corresponding function,  $\zeta(x)$.

\begin{figure}[h!]
\center{\includegraphics[width = .8\textwidth]{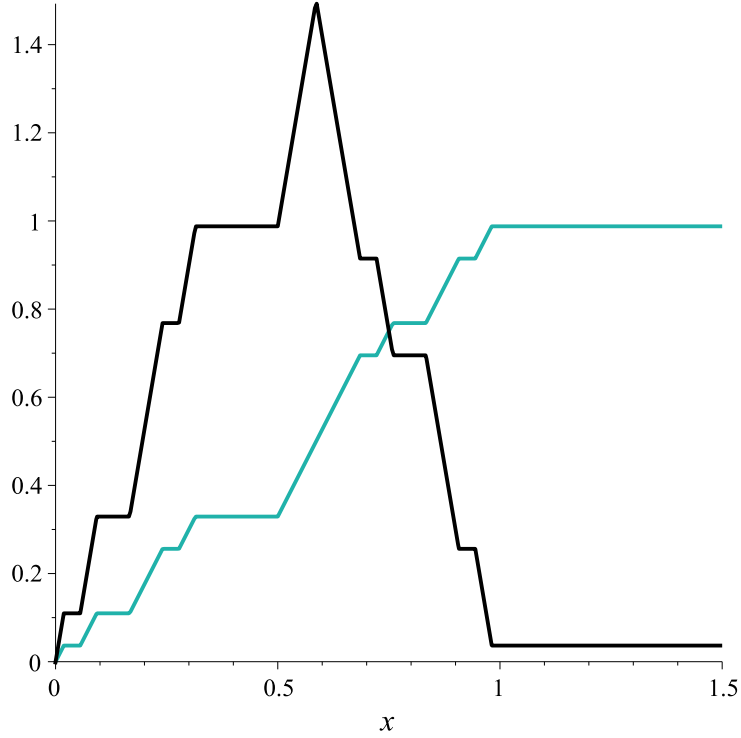} \\ b)}
\caption{Graphs of the functions $F(x)$ (black) and $\zeta (x)$ (blue) when $H=3$ and $T=4$.  Since $F(x) = f(\zeta(x))$, $F$ and $\zeta$ are constant on the same intervals, and $F$ is increasing when $\zeta$ is less than $\frac12$, decreasing when $\zeta$ is greater than $\frac12$. } \label{f13}
\end{figure}

On the interval $[0,\,1],$ there are $2^{T-1}$ disjoint intervals on which the function $F(x)$ is constant. This means that there are $2^{T-1}$ 
periodic points, where $F'(x)=0.$ These points are $T$-periodic points of the map $f(x)$ with negative multipliers. Generally speaking, 
the proper period of these points may be less than $T$ if $T$ is not a prime number. The question of stability of these so-called subcycles 
is considered in the next section. The fact that the derivative is zero means that the sequence defined by iterating Equation~(\ref{ceq}) 
 converges very quickly to these cyclic points, i.e. they are {\it super-stable}.   In particular, if $x_0 \geq \hat{x},$ then $x_1 = F(x_0) = \frac{H^T}{H^T + 1} > \frac12,$
$x_2 = F(x_1) = f(x_1) = \frac{H}{H^T + 1} < \frac12,$ \ldots , $x_T = F(x_{T-1}) = f(x_{T-1}) = \frac{H^{T-1}}{H^T + 1} < \frac12,$
$x_{T+1} = F(x_T) = f(x_T) = \frac{H^T}{H^T + 1} = x_1,$ i.e., $\left\{ x_1,\ldots,\,x_T\right\}$ is a $T$-cycle of both Equations~(\ref{eq}) and 
(\ref{ceq}). 

\section{Stabilization of Subcycles} \label{sec:Subcycles}

The system given by Equation \eqref{eq} has $T$-cycles for any $T \geq 1$.  Moreover, it is not difficult to calculate the number of $T$-cycles of any given period. Since $T$-periodic points are fixed points of $f^T(x)$, they will be points of intersection of the graph of $y = f^T(x)$ and the graph $y = x$.  These two graphs intersect exactly $2^T$ times for $x\in[0,\,1]$. Two points of intersection correspond to 
the fixed points of $f$ at $x=0,$ and $x = 1 - \frac{1}{H+1}$. 
If $T$ is a prime number, then there are exactly $\frac{2^T - 2}{T}$ cycles of length $T$. (By Fermat's Little Theorem, if $T$ is prime, $2^T-2$ is divisible by T, and thus, $\frac{2^{T}-2}{T}$ is an integer.) If $T = \prod\limits_{j=1}^{s} {\tau_{j}^{\rho_j}},$ where $\tau_1,\ldots,\,\tau_s$ are distinct prime numbers, 
then there are exactly 
$$\frac{1}{T} \left[{2^T - \sum\limits_{j=1}^{s}{2^\frac{T}{\tau_j}} + \sum\limits_{\substack{i,\,j=1 \\ i<j}}^{s}{2^\frac{T}{\tau_i \tau_j}} + \ldots +
(-1)^s 2^\frac{T}{\tau_1 \cdot \ldots \cdot \tau_s}}\right]$$
cycles of the length $T.$ The fact that the given fraction is an integer is a special case of 
Gauss's theorem, see for example \cite[p.84]{Dixon}.

Let $\tau$ be a factor of the number $T$, and let the $T$-cycle of the system given by Equation~(\ref{ceq}) 
be locally asymptotically stable. \textit{The task} is: to find out which $\tau$-cycles of the same system will be locally asymptotically stable?  It turns out that the answer depends on the parity of $T/\tau$.

\begin{theorem}\label{t4}
Let condition~(\ref{inn}) be satisfied, i.e. $T$-cycles of equation~(\ref{eq}), for which the multipliers are positive, are locally asymptotically stable cycles 
of equation~(\ref{ceq}). Then all the $\tau$-cycles of equation~(\ref{eq}), for which the multipliers are positive, and all the $\tau$-cycles of 
equation~(\ref{eq}), for which the multipliers are negative and the number $\frac{T}{\tau}$ is even, are locally asymptotically stable cycles of 
equation~(\ref{ceq}).

Let condition~(\ref{inn}) be satisfied, i.e. $T$-cycles of Equation~(\ref{eq}), for which the multipliers are negative, are locally asymptotically stable cycles 
of Equation~(\ref{ceq}). Then all the $\tau$-cycles of Equation~(\ref{eq}), for which the multipliers are negative and the number $\frac{T}{\tau}$ is odd, 
are locally asymptotically stable cycles of Equation~(\ref{ceq}).
\end{theorem}

\begin{proof}
Let $\left\{\eta_1,\ldots,\,\eta_T\right\}$ be a $T$-cycle and $\left\{\widehat{\eta}_1,\ldots,\,\widehat{\eta}_\tau\right\}$ a $\tau$-cycle of 
the system given by Equation~(\ref{eq}). Let $\mu_T = f'(\eta_T)\cdot\ldots\cdot f'(\eta_1),$ $\mu_\tau = f'(\widehat{\eta}_\tau)\cdot\ldots\cdot f'(\widehat{\eta}_1)$ 
be the corresponding multipliers of these cycles. These cycles will also be  cycles of Equation~(\ref{ceq}). We assume that the condition for local 
asymptotic stability of the $T$-cycle of Equation~(\ref{ceq})
\begin{equation}\label{tc}
\left| \mu_T \left( \vartheta + (1 - \vartheta) \mu_T\right)^T\right| < 1 
\end{equation}
is satisfied. Let us find the multiplier of the $\tau$-cycle of equation~(\ref{ceq}). Let, as before, 
$F(x) = f \left( \vartheta x + (1-\vartheta) f^{(T)}(x)\right),$ and $p=\frac{T}{\tau}.$ Calculate
\begin{gather*}
F'(\widehat{\eta}_j) = f'(\widehat{\eta}_j) \left( \vartheta + (1-\vartheta) \left( f'(\widehat{\eta}_\tau)\cdot\ldots\cdot f'(\widehat{\eta}_1) \right)^p \right) = \\
= f'(\widehat{\eta}_j) \left( \vartheta + (1-\vartheta) (\mu_\tau)^p \right),\; j=1,\ldots,\,\tau.
\end{gather*}
Then the multiplier of the $\tau$-cycle of equation~(\ref{ceq}) equals 
$F'(\widehat{\eta}_\tau)\cdot\ldots\cdot F'(\widehat{\eta}_1) = \mu_\tau \left( \vartheta + (1-\vartheta) (\mu_\tau)^p \right)^\tau$, and, accordingly, 
the condition for local asymptotic stability of the $\tau$-cycle of equation~(\ref{ceq}) is:
\begin{equation}\label{tauc}
\left| \mu_\tau \left( \vartheta + (1 - \vartheta) (\mu_\tau)^p \right)^\tau \right| < 1 
\end{equation}

Since $\mu_T = \pm H^T,$ $\mu_\tau = \pm H^\tau,$ the following cases are possible:

\begin{enumerate}
\item[a)] $\mu_T<0,$ $\mu_\tau<0,$ then 
$\left\{\begin{array}{ll}
\mu_T = (\mu_\tau)^p,\, p \text{ is odd}, \\
\mu_T = -(\mu_\tau)^p,\, p \text{ is even},
\end{array}\right.$

\medskip
\item[b)] $\mu_T<0,$ $\mu_\tau>0,$ then $\mu_T = -(\mu_\tau)^p$ for any $p,$  

\medskip
\item[c)] $\mu_T>0,$ $\mu_\tau<0,$ then 
$\left\{\begin{array}{ll}
\mu_T = -(\mu_\tau)^p,\, p \text{ is odd}, \\
\mu_T = (\mu_\tau)^p,\, p \text{ is even},
\end{array}\right.$

\medskip
\item[d)] $\mu_T>0,$ $\mu_\tau>0,$ then $\mu_T = (\mu_\tau)^p$ for any $p.$  
\end{enumerate}

Inequality~(\ref{tc}) implies inequality~(\ref{tauc}) if and only if one of the conditions holds: $\mu_T<0,$ $\mu_\tau<0,$ $p$ is odd; 
$\mu_T>0,$ $\mu_\tau<0,$ $p$ is even; $\mu_T>0,$ $\mu_\tau>0$ for any $p.$ The first condition is possible for the parameters $\vartheta$ that satisfy 
inequalities~(\ref{inp}), the second and third conditions are possible for the parameters $\vartheta$ that satisfy inequalities~(\ref{inn}), whence the conclusion 
of the Theorem follows.
\end{proof}

Consider again the example of stabilization of 5-cycles from Section \ref{sec:GraphicalStudy}.  Then, besides locally 
asymptotically stable 5-cycles, the fixed points at $\eta=0$ when $\vartheta = \frac{H^T + \frac{0.4}{H}}{H^T - 1}$ and $\eta=0.8$ when
$\vartheta = \frac{H^T - \frac{0.4}{H}}{H^T + 1}$ will also be stable. The initial points: $x_0=0.001$ and $x_0=0.801$ are in the basins of attraction of the corresponding fixed points, and their orbits will therefore converge to the fixed points rather than to one of the true period-5 cycles.  Periodic and fixed points are shown in Figure~\ref{f14}.

\begin{figure}[ht!]
\centering
\includegraphics[scale=0.28]{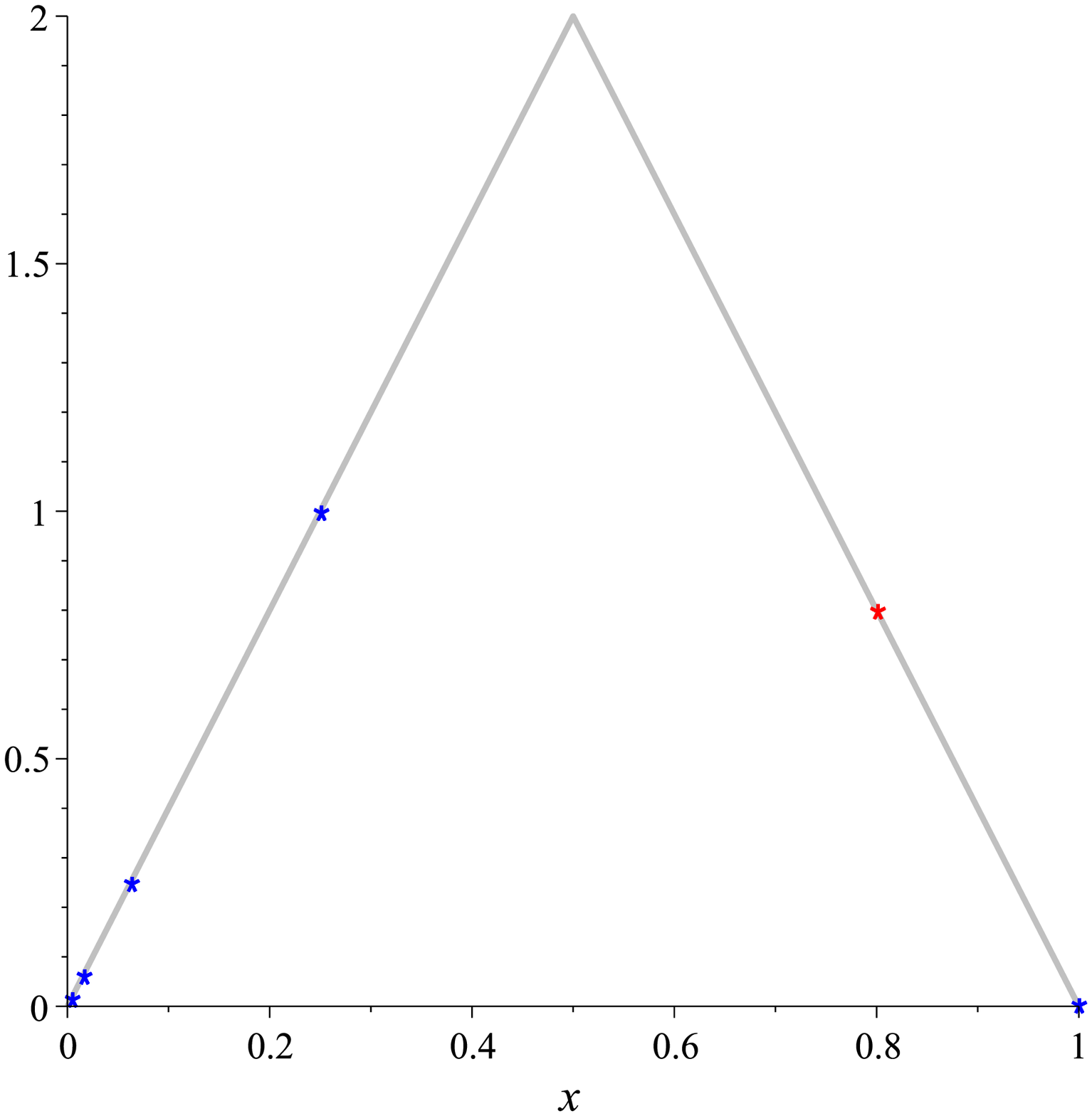}
\hspace{1cm}
\includegraphics[scale=0.28]{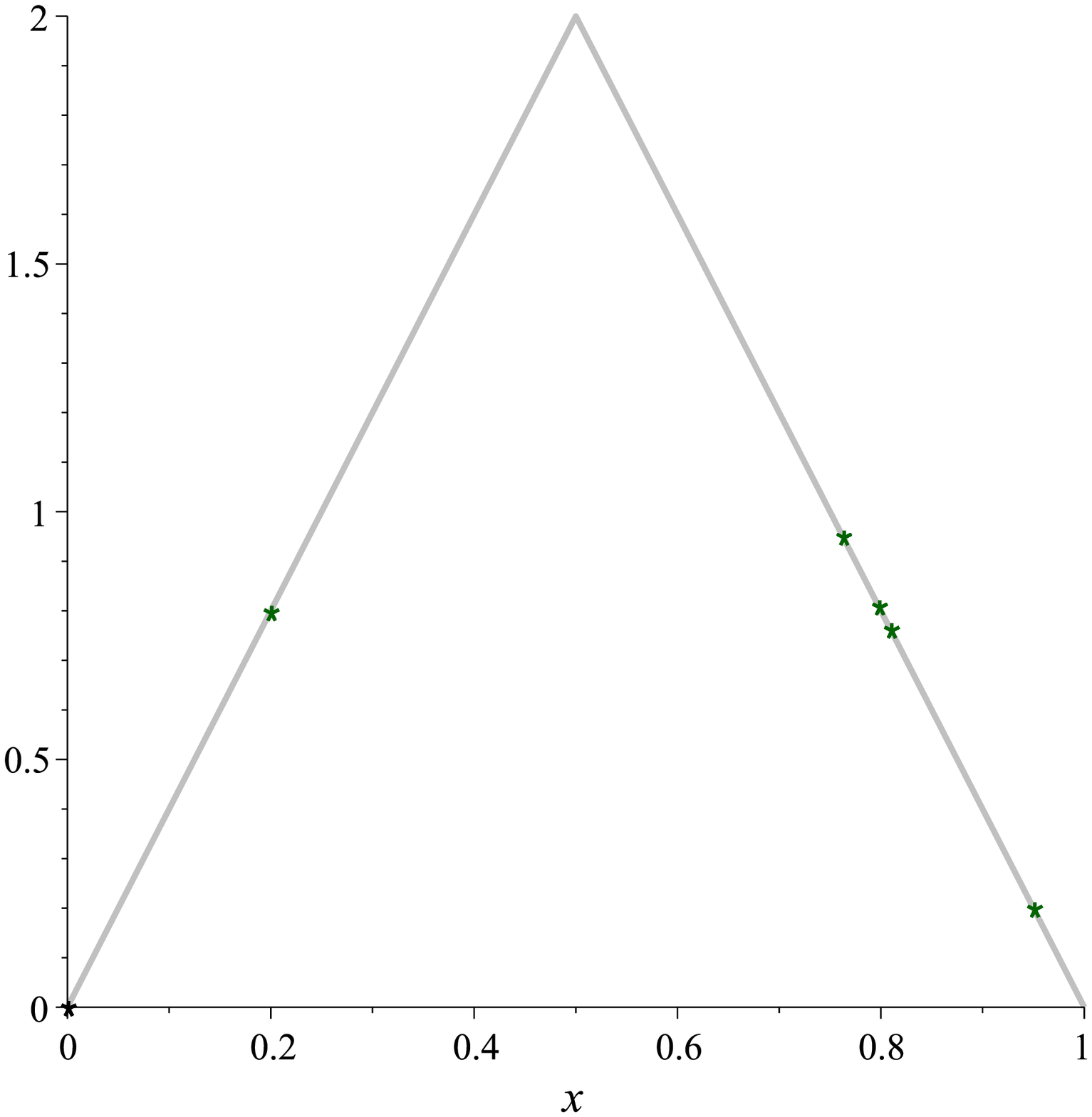}
\caption{Points of 5-cycles and fixed points of the tent map where $H = 4$. In the Left Panel, a 5-cycle with positive multiplier (blue) and the non-zero fixed point (red) are stabilized using $\vartheta = \frac{H^T + \frac{0.4}{H}}{H^T - 1} $.  In the Right Panel, a 5-cycle with negative multiplier (green) and the fixed point at 0 are stabilized using $\vartheta = \frac{H^T - \frac{0.4}{H}}{H^T + 1}$.  To put this example in the context of Theorem \ref{t4}, we note that in this case $T = 5$ and $\tau = 1$, so that $p = T/\tau = 5$ is odd.}\label{f14}
\end{figure}

\section{Distribution of cyclic points and visualization of the Cantor set} \label{sec:DistributionCyclicPoints}

As noted above, the invariant set of equation~(\ref{eq}) at $H=3$ is the classical Cantor middle thirds set. However, due to its strong instability, it is impossible to visualize 
the points of this set using equation~(\ref{eq}). Recall that the Cantor set can be defined as those real numbers in $[0,1]$ whose ternary expansions consist   only of 0's and 2's.  The Cantor set is characterized by two types of points: the points that are end-points of the open intervals that are adjacent 
to the Cantor set are called points of the first type (these are points which terminate in all 0's or all 2's when written in ternary expansion - they take the form $\frac{p}{3^k}$ for some natural numbers $p$ and $k$), while all the other points of the set are called points of the second type. Points of the first type are solutions of the equations $f^{(k)}(x) = 1,$ for $k=1,\,2,\ldots\,$. The set of points of the first type is countable. When $x<1/2$ and in the Cantor set (so that is, it's first digit in ternary expansion is 0), the tent map acts as a right shift map on the ternary expansion.  If $x>1/2$, the tent map acts on the ternary expansion by swapping the 2's and 0's, and then acting as a shift map. Among the points of the second type, we can select
a subset consisting of all periodic points of the system given by Equation ~\eqref{eq}. They are obtained as the union of all the roots of the equations $f^{(k)}(x) = x,$ $k=1,\,2,\ldots\,$. 
The set of periodic points of the map (\ref{tent}) is also countable.  Points whose ternary expansion is periodic will themselves be periodic because of the way the tent map acts as a shift on the ternary expansions.  Thus, there are periodic points arbitrarily close to points of the first type of the Cantor set. 

The following question arises: how are the periodic points of one orbit of a given period $T$  distributed through the Cantor set?  More precisely, how uniformly 
do the periodic points of the orbit of a given period $T$ fill the Cantor set? We do not consider the analytical solution of this problem in this paper, 
but we provide some examples simulating the density functions for the distribution of periodic point, and we compare them graphically  with 
the analogous function for a random sample of elements from the set of points of the first type in the Cantor set.

Let us take for example the period $T=1009,$ with an accuracy $10^{-525}$, initial value $x_0=0.555$, and values for the control parameter
$\vartheta = \frac{H^T \pm \frac{0.6}{H}}{H^T + 1}.$ We will get $2T=2018$ cyclic points. We then simulate the density function for the distribution 
of the cyclic points set (Figure~\ref{f15}-a). The graph shows that the estimated periodic points are not quite evenly distributed in the Cantor set. Let us now find \textit{two hundred} orbits with the period 
$T=1009,$ i.e., we get 20180 cyclic points. Surprisingly, the graph of the  simulated density function for the new set of points (Figure~\ref{f15}-b) does not differ much 
from the plot in the previous case. Finally, let us plot the density function for the distribution of randomly selected 200000 points of the first type of the Cantor set.
For this purpose, represent a subset of the points of the first type of the Cantor set in their ternary expansion $s = \sum\limits_{j=1}^{N}{\alpha_j}{\frac{2}{3^j}},$ where 
$\alpha_j \in \{ 0, 1 \}$.  If $N=25$, the maximum number of possible points is $2^{25}$. 
Let us randomly choose values for $\alpha_j$: either zero or one, and thus construct 200000 points. Next, we  graph the distribution  
of the resulting set (Figure~\ref{f16}). Comparison of the graphs in Figures~\ref{f15}, \ref{f16} shows that the points of the first type, constructed by the method outlined above, 
are distributed more evenly on the Cantor set. 
 
\begin{figure}[h!]
\begin{minipage}[h]{0.45\linewidth}
\center{\includegraphics[scale=0.28]{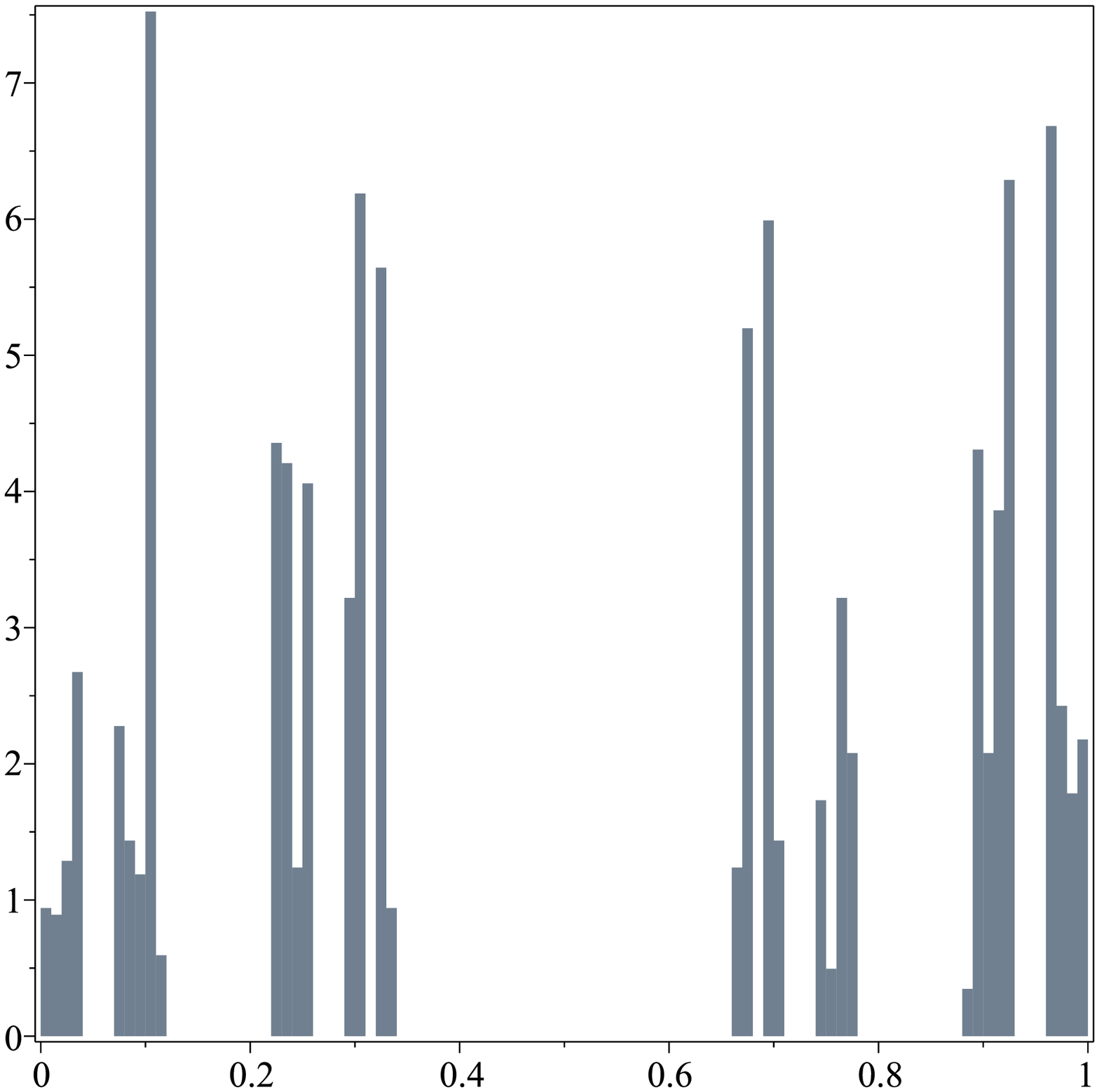} \\ a)}
\end{minipage}
\hspace{1cm}
\begin{minipage}[h]{0.45\linewidth}
\center{\includegraphics[scale=0.28]{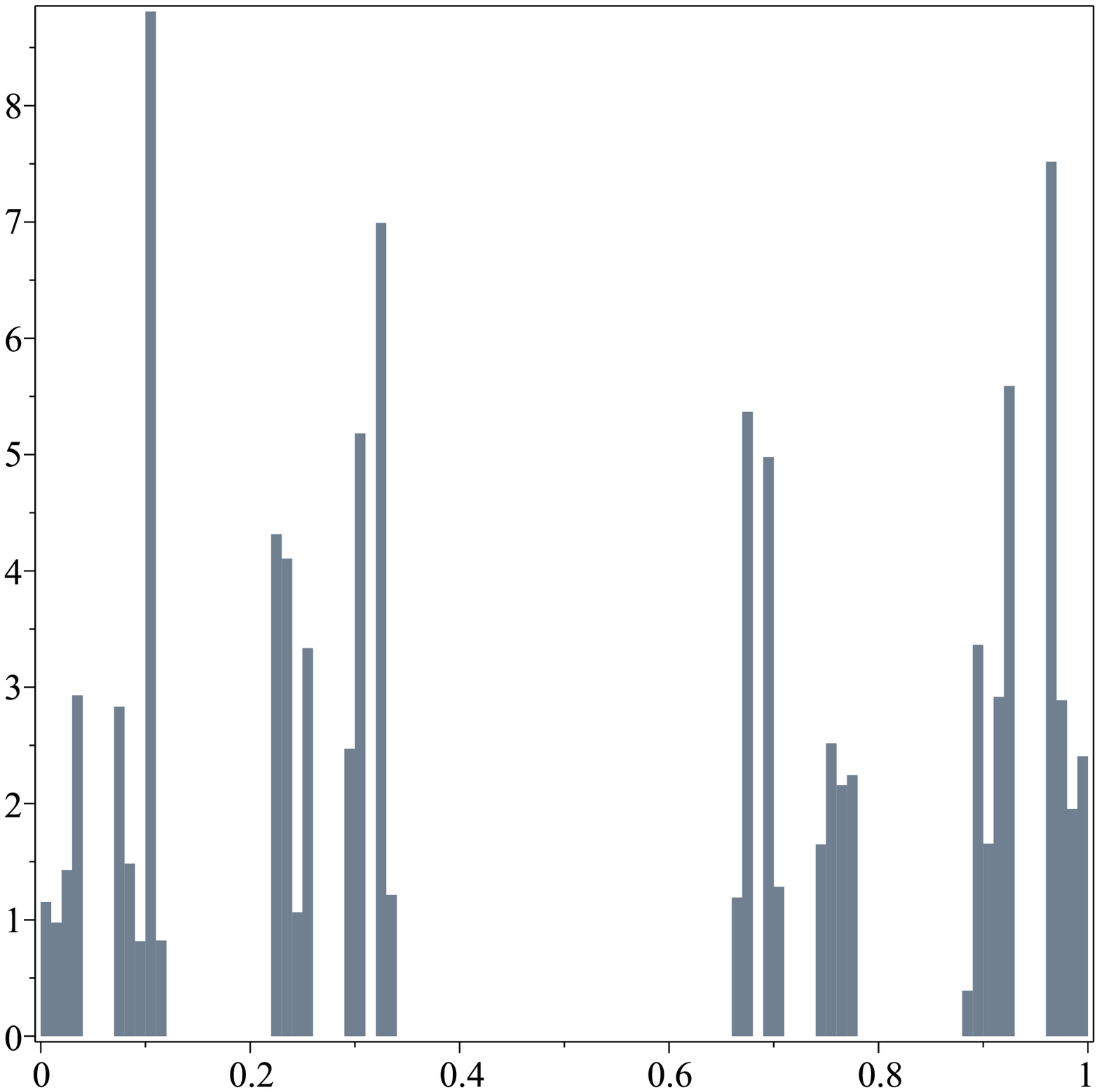} \\ b)}
\end{minipage}
\caption{The graph of the distribution density function of the set of cyclic points of two and two hundred 1009-periodic orbits of map~(\ref{tent}) at $H=3$} \label{f15}
\end{figure}

\begin{figure}[h!]
\centering
\includegraphics[scale=0.28]{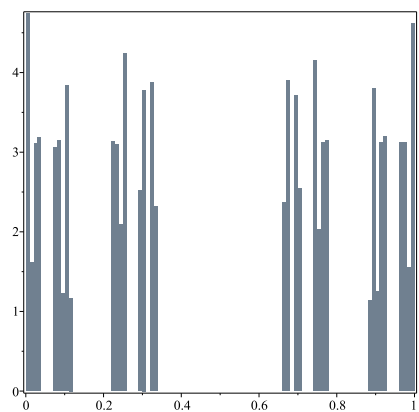}
\caption{Graph of the distribution density function of the set of 200000 random points of the first type of the Cantor set.  Note that the distribution is more uniform across the bins of the histogram.} \label{f16}
\end{figure}

\section{Conclusions} \label{sec:Conclusions}


A general predictive control framework was developed in \cite{DSI} for finding a given set of periodic orbits by making them locally stable.  This suggests a solution to the problem of numerically describing invariant sets \footnote{We consider the {\it largest} bounded set $\Gamma$, such that $f(\Gamma) \subseteq \Gamma$} of nonlinear dynamical systems, since periodic orbits are often dense in these invariant sets.

Since the periodic orbits themselves are repelling sets the problem is separated in two parts:
a local stabilization of periodic orbits of a given period, which could be fairly large, and then selection of an initial point in the basin of attraction of the stabilized periodic orbit.

If the invariant set is a global attractor, then all trajectories are bounded, and therefore most initial points are in the basin of attraction of one of the stabilized periodic orbits.

 If the invariant set is repelling, then the analogous problem turns out to be much more complicated: simple local stability 
of the controlled orbit is not sufficient, due to the fact that its basin of attraction can have a small measure, 
and a very complicated structure.  For example, it might not be simply connected, or it may have a fractal boundary. The orbits of most initial points will go to infinity. Therefore, the problem of choosing an initial point in the basin of attraction of a stabilized orbit becomes significant.

However, the method of generalized predictive control, used to stabilize the orbit, has an important characteristic: in addition to the local asymptotic stability 
of the orbits of controlled system, the rest of the orbits remain bounded for a sufficiently large set of initial values. In this paper we showed that, for the generalized tent map, it is possible to ensure 
that all solutions are bounded. The global behavior of solutions for the generalized logistic map, the generalized Lozi \cite{Lozi1}, Hénon \cite{Henon}, Ikeda \cite{Ikeda}, Elhadj-Sprott \cite{El-Spr} maps, etc. is somewhat more complicated. 
For such systems, we conjecture that it is necessary to choose the control parameter as a function of the current state of the system. 
Investigation of the global behavior of 
the controlled systems solutions for these maps is the task for future research.

\section{Acknowledgment}
The authors wish to thank Andrew Sills for his help in finding the reference to Gauss' theorem. 

\bibliographystyle{unsrt}
\bibliography{tentbib}

\end{document}